\documentclass{article}
\usepackage{latexsym}
\usepackage{amssymb}
\usepackage{amsmath}
\usepackage{color}
\usepackage{graphicx}
\usepackage{amsthm}
\usepackage{indentfirst}
\usepackage{mathrsfs}
\usepackage{enumerate}
\allowdisplaybreaks

\newtheorem{theorem}{\color{black}\indent Theorem}[section]
\newtheorem{lemma}{\color{black}\indent Lemma}[section]

\newtheorem{remark}{\color{black}\indent Remark}[section]

\renewcommand{\baselinestretch}{1.2}

\DeclareMathOperator{\diag}{{diag}}
\DeclareMathOperator{\meas}{{Meas}}
\DeclareMathOperator{\spec}{{Spec}}
\begin{document}
\large
\title{Averaging method for quasi-periodic response solutions}
\author{Jiamin Xing$^{a}$, ~ {Yong Li$^{b,a}$}, ~ {Shuguan Ji$^{a}$}\thanks{Corresponding author.
\newline\hspace*{0.3cm} E-mail addresses: xingjiamin1028@126.com (J. Xing), liyongmath@163.com (Y. Li), jisg100@nenu.edu.cn (S. Ji).}\\{$^{a}$School of Mathematics and Statistics, and}\\
{Center for Mathematics and Interdisciplinary Sciences,}\\
{Northeast Normal University, Changchun 130024, P. R. China.}\\
{$^{b}$College of Mathematics, Jilin University,}
\\{Changchun, 130012, P. R. China.}\\
 }
\date{}

\maketitle
\par
{\bf Abstract.}
In this paper, we present an averaging method for obtaining quasi-periodic response solutions in perturbed, real analytic, quasi-periodic systems with Diophantine frequency vectors.
Under the assumptions that the averaged system possesses a non-degenerate equilibrium and that the eigenvalues of its linearized matrix are pairwise distinct, we show that the original system admits a quasi-periodic response solution for parameters in a Cantorian set.
The proof relies on KAM techniques. It is worth mentioning that our results do not require the equilibrium to be hyperbolic, meaning that the eigenvalues of the linearized matrix of the averaged system may be purely imaginary. Furthermore, the proposed averaging method is applicable to second-order systems, and a higher-order averaging framework is also established.


\section{Introduction and main results}
The averaging method  developed by Krylov and Bogolyubov \cite{K-B} and Bogolyubov \cite{B}
establishes a quantitative relationship between the solutions of a perturbed non-autonomous system
and its averaged system. As Hale stated in \cite{Hale}, the averaging method is one of the most important methods for determining periodic and almost periodic solutions of nonlinear differential equations that contain a small parameter.

Consider the system
\begin{align}\label{01.1}
\dot{x}=\varepsilon f(t,x,\varepsilon),
\end{align}
and its corresponding first-order averaged system
$$\dot{x}=\varepsilon \overline{f_0}(x),$$
where $\varepsilon\in\mathbb{R}$ is a small parameter, and $f:\mathbb{R}\times\mathbb{R}^n\times\mathbb{R}\rightarrow\mathbb{R}^n$ is a smooth function that is almost periodic in $t$.
The first-order averaged function $\overline{f_0}$ is defined as
$$\overline{f_0}(x)=\lim\limits_{T\rightarrow +\infty}\frac{1}{T}\int_0^Tf(t,x,0)\mathrm{d}t.$$
The classical averaging method states that if the first-order averaged system has a hyperbolic equilibrium $x_0$
(meaning $\overline{f_0}(x_0)=0$ and all eigenvalues of matrix $\partial_x\overline{f_0}(x_0)$ have nonzero real parts),
then the original system \eqref{01.1} has a hyperbolic almost periodic solution $x^*(t,\varepsilon)$. This solution shares the same stability type as $x_0$, and its module satisfies $m(x^*(\cdot,\varepsilon))\subset m(f(\cdot,x,\varepsilon))$; see  Theorem 3.1 in \cite[Chapter 5]{Hale}.

When $f$ is $T$-periodic, provided that $x_0$ is non-degenerate, i.e., $\det(\partial_x\overline{f_0}(x_0))\neq0$, the existence of a $T$-periodic solution for system \eqref{01.1} can still be guaranteed; see \cite{GH,Hale}.
If the first-order averaged function vanishes, the existence of periodic solutions for system \eqref{01.1} can be determined by higher-order averaged functions.
Through a sequence of $T$-periodic near-identity transformations, system \eqref{01.1} can be transformed into the form:
\begin{align*}
\dot{y}=\sum_{i=1}^N\varepsilon^i Y_i(y)+\varepsilon^{N+1}G( t,y,\varepsilon),
\end{align*}
where $Y_1=\overline{f_0}$, and $Y_2, \dots, Y_N$ can be computed recursively (see, for instance, \cite{S-V-M}). Each function $Y_i$
is called the averaged function of order $i$. Consider the so-called guiding system:
\begin{align*}
\dot{y}= Y_k(y),
\end{align*}
where $Y_k$ is the first non-vanishing averaged function, that is, $Y_i= 0$ for $i\in\{1,\dots,k-1\}$ and $Y_k\neq0$.
If $Y_k$ has an isolated equilibrium point, the existence of a periodic solution for system \eqref{01.1}  can still be established via topological tools \cite{BL,LNT}.

For a non-hyperbolic equilibrium point, the stability of the periodic solution depends on higher-order characteristics of the equilibrium, such as its first Lyapunov coefficient. In such cases, an invariant torus generically bifurcates
from the periodic orbit through a secondary Hopf bifurcation \cite{CN}.
If the vector field of \eqref{01.1} satisfies certain geometric conditions, the computation of higher-order averaged functions can be simplified, and the non-degeneracy condition can  be relaxed, as demonstrated in \cite{NS,XYL}.
For more recent developments and applications of the averaging method for periodic solutions, we refer to \cite{ALP,BLV,GLWZ,LHL,N} and the references therein. In particular, Novaes and Pereira \cite{NP} and Pereira et al. \cite{PNC} developed a higher-order averaging method for the existence of invariant tori, showing that if the guiding system possesses a hyperbolic limit cycle, then the original periodic system admits an invariant torus.

For a quasi-periodic system, a quasi-periodic solution that shares the same basic frequencies as the system is called a response solution. The existence of response solutions in quasi-periodic systems has been extensively studied; see, for example, \cite{BB,B-W,B-B,GV,SXY,XSS}
and the references therein. Since quasi-periodic systems can be regarded as a special class of almost periodic systems, and the module containment relationship $m(x^*(\cdot,\varepsilon))\subset m(f(\cdot,x,\varepsilon))$ implies that in the quasi-periodic case, the solution $x^*(t,\varepsilon)$ and the vector field $f(t,x,\varepsilon)$ share the same basic frequencies. Hence, the classical averaging method establishes the existence of response solutions in the hyperbolic case.

In the non-hyperbolic case, the linearization operator of the system is non-invertible due to the presence of ``small divisors", rendering the proof method used in the hyperbolic case inapplicable.
The main purpose of this paper is to provide a general averaging method for determining response solutions in the real analytic quasi-periodic system \eqref{01.1}, allowing the matrix $\partial_x\overline{f_0}(x_0)$ to have purely imaginary eigenvalues. Additionally, a corresponding higher-order averaging method is presented.

Let $U$ be an open bounded subset of $\mathbb{R}^n$ and let $\mathbb{T}^d$ denote the $d$-dimensional torus, where $d\in\mathbb{N}_+$. Consider the system
\begin{align}\label{1.1}
\dot{x}=\varepsilon f(\omega t,x,\varepsilon),
\end{align}
where $0<\varepsilon<\varepsilon_0$ is a small parameter for some constant $\varepsilon_0>0$, $\omega\in\mathbb{R}^d$ is the basic frequency vector,  and
$f:\mathbb{T}^d\times U\times[0,\varepsilon_0)\rightarrow \mathbb{R}^n$ is continuous and real analytic on  $\mathbb{T}^d\times U$ for each $\varepsilon\in[0,\varepsilon_0)$. Furthermore, assume that $f(\theta,x,\varepsilon)$ and $\partial_x f(\theta,x,\varepsilon)$ are Lipschitz continuous with respect to $\varepsilon$ on $[0,\varepsilon_0)$ uniformly for $(\theta,x)\in\mathbb{T}^d\times U$. Let
$f_0(\omega t,x)=f(\omega t,x,0)$ and
$$\overline{f_0}(x)=\lim\limits_{T\rightarrow +\infty}\frac{1}{T}\int_0^Tf_0(\omega t,x)\mathrm{d}t.$$

Assume $\omega$ satisfies the Diophantine condition:
\begin{align}\label{DC}
|\langle k,\omega\rangle|\geq\gamma|k|^{-\tau},\ \ \ \forall\ k\in\mathbb{Z}^d\backslash\{0\},
\end{align}
where $\gamma>0$, $\tau>d-1$ are fixed constants, and $|k|=|k^{(1)}|+\dots+|k^{(d)}|$ with $k^{(j)}$ denoting the $j$-th component of $k$ for $1 \leq j \leq d$.

Our first main result is the following theorem.
\begin{theorem}\label{theorem}
Suppose that $\omega$ satisfies the Diophantine condition \eqref{DC} and that
there exists a point $x_0\in U$ such that $\overline{f_0}(x_0)=0$ and the eigenvalues of the matrix $\partial_x\overline{f_0}(x_0)$ are all nonzero and pairwise distinct. Then
there exist a constant $0<\varepsilon_1<\varepsilon_0$ and a  Cantorian set $\mathcal{E}\subset(0,\varepsilon_1)$ such that for each $\varepsilon\in\mathcal{E}$, the system
\eqref{1.1} has a quasi-periodic response solution $x^*(t,\varepsilon)$ satisfying $\lim\limits_{\substack{\varepsilon \in \mathcal{E} \\ \varepsilon \rightarrow 0}}x^*(t,\varepsilon)=x_0$ uniformly for $t\in\mathbb{R}$. Moreover, the relative Lebesgue measure of $\mathcal{E}\cap(0,\varepsilon)$ with respect to $(0,\varepsilon)$ tends to 1 as $\varepsilon\rightarrow0$.
\end{theorem}

We prove Theorem \ref{theorem} by establishing the existence of a response solution for the system:
\begin{equation}\label{01.5}
\dot{x}=\varepsilon^a f(\omega t,x)+\varepsilon^bg(\omega t,x,\varepsilon),
\end{equation}
where $b>a>0$ are constants,
$f:\mathbb{T}^d\times U\rightarrow \mathbb{R}^n$ is real analytic, and $g:\mathbb{T}^d\times U\times(0,\varepsilon_0)\rightarrow \mathbb{R}^n$ is continuous and real analytic on  $\mathbb{T}^d\times U$ for each $\varepsilon\in(0,\varepsilon_0)$. Moreover, we assume that $g(\theta,x,\varepsilon)$ and $\partial_x g(\theta,x,\varepsilon)$ are Lipschitz continuous with respect to $\varepsilon$ on $(0,\varepsilon_0)$ uniformly for $(\theta, x) \in \mathbb{T}^d \times U$.

Indeed, through reparameterization, system \eqref{1.1} can be rewritten as system \eqref{01.5}. Consequently, the following result extends to system \eqref{01.5}, which enhances the flexibility of the method in applications.
\begin{theorem}\label{th1.2}
Suppose that $\omega$ satisfies the Diophantine condition \eqref{DC}, and that
there exists a point $x_0\in U$ such that $\overline{f}(x_0)=0$ and the eigenvalues $\lambda^{(1)},\dots,\lambda^{(n)}$ of the matrix $\partial_x\overline{f}(x_0)$ are nonzero and pairwise distinct. Then the conclusion of Theorem \ref{theorem} holds for system
\eqref{01.5}.
\end{theorem}

Jorba and Sim\'{o} \cite{Jo} established the existence of response solutions for the $n$-dimensional real analytic system
\begin{equation}\label{01.6}
\dot{x}=(A+\varepsilon Q(\omega t,\varepsilon))x+\varepsilon g(\omega t,\varepsilon)+h(\omega t,x,\varepsilon),
\end{equation}
where $h(\omega t,x,\varepsilon)=\mathcal{O}(|x|^2)$ as $x\rightarrow0$.  The analysis was based on the following assumptions:
the eigenvalues $\lambda^{(1)}_0,\dots,\lambda^{(n)}_0$
of $A$ are nonzero and pairwise distinct; the non-resonance conditions
$$|\langle k,\omega\rangle-\lambda^{(i)}_0|\geq\gamma|k|^{-\tau}\ {\rm and}\  |\langle k,\omega\rangle-\lambda^{(i)}_0+\lambda^{(j)}_0|\geq\gamma|k|^{-\tau}$$ hold for all $k\in\mathbb{Z}^d\backslash\{0\}$ and $1\leq i, j\leq n$;
and  the eigenvalues of $\underline{A}(\varepsilon)$ are bi-Lipschitz continuous with respect to $\varepsilon$, that is, there exist constants $L_2 \geq L_1 > 0$ such that for all sufficiently small $\varepsilon_1, \varepsilon_2$,
$$L_1|\varepsilon_1-\varepsilon_2|\leq|\lambda^{(i)}(\varepsilon_1)-\lambda^{(j)}(\varepsilon_1)
-(\lambda^{(i)}(\varepsilon_2)-\lambda^{(j)}(\varepsilon_2))|\leq L_2|\varepsilon_1-\varepsilon_2|, \quad 1\leq i< j\leq n,$$
 and
$$L_1|\varepsilon_1-\varepsilon_2|\leq|\lambda^{(i)}(\varepsilon_1)
-\lambda^{(i)}(\varepsilon_2)|\leq L_2|\varepsilon_1-\varepsilon_2|, \quad 1\leq i\leq n.$$
Here,  $\underline{A}(\varepsilon)=A+\varepsilon \overline{Q}(\varepsilon)+\overline{\partial_{x}h(\underline{x}(t,\varepsilon),t,\varepsilon)}$, with $\underline{x}(t,\varepsilon)$ being the unique quasi-periodic solution of $\dot{x}=Ax+\varepsilon g(\omega t,\varepsilon)$.
Under these conditions, they proved that for $\varepsilon$ in a Cantorian set of nearly full measure, the system admits a response solution.

When the matrix $A$ is singular, the system \eqref{01.6} is called degenerate. For $1$-dimensional degenerate systems, Xu and Jiang \cite{XJ}, Si and Si \cite{SS}, and
Cheng, de la Llave and Wang \cite{CdW} studied response solutions for the following equation
\begin{equation*}
\dot{x}=x^{l}+h(\omega t,x)+\varepsilon f(\omega t, x),\ \ \ x\in\mathbb{R},
\end{equation*}
where $l>1$ is an integer and $h(\omega t,x)=\mathcal{O}(|x|^{l+1})$ as $x\rightarrow0$. For response solutions in $2$-dimensional degenerate systems under quasi-periodic perturbations, see \cite{MQX,SY,ZJS}; for $3$-dimensional systems, see \cite{XLW}. The systems \eqref{1.1}
and \eqref{01.5} considered in this paper can be viewed as high-dimensional degenerate
systems, corresponding to the case of $A = 0$ in \eqref{01.6}.

In this paper, we also propose a higher-order averaging method for response solutions based on Theorem \ref{th1.2}.
We consider the higher-order perturbed system
\begin{align*}
\dot{x}=\sum\limits_{k=1}^N\varepsilon^k f_k(\omega t,x)+\varepsilon^{N+1}g(\omega t,x,\varepsilon),
\end{align*}
where $N$ is a fixed positive integer. For each $k\in\{1,\dots, N\}$, the function $f_k$ satisfies the same assumptions as $f$ in system \eqref{01.5},
and $g$ satisfies the corresponding assumptions therein.
We establish that the conclusion of Theorem \ref{theorem} remains valid, provided that the first non-vanishing higher-order averaged function of the system satisfies the same conditions as $\overline{f_0}$ in Theorem \ref{theorem}. A detailed treatment of this result is provided in Section \ref{sec5}.

Our work naturally prompts the question of whether the averaging method developed herein can be generalized to the broader class of almost-periodic systems. While the philosophical core of our method may remain applicable, such an extension faces significant technical hurdles, primarily due to the fact that the frequency module of a general almost-periodic function is not necessarily finitely generated. This introduces profound challenges in controlling the small divisors and in the constructive scheme of normal forms. A detailed investigation of this generalization is beyond the scope of the present work but represents a challenging direction for future research.

The paper is organized as follows. Section \ref{sec2} applies Theorems \ref{theorem} and \ref{th1.2} to investigate the existence of response solutions for second-order perturbed quasi-periodic and homogeneous systems. The proof of Theorem \ref{th1.2} is presented in Section \ref{sec3} and is divided into five parts. Subsequently, Section \ref{sec4} provides the proof of Theorem \ref{theorem}. In Section \ref{sec5}, we develop a higher-order averaging method for response solutions in quasi-periodic systems. Finally, the appendix collects the fundamental lemmas used in our proofs.

\section{Applications}\label{sec2}
\subsection{Averaging method for second-order systems}
As a consequence of  Theorem \ref{theorem}, we derive an averaging method for response solutions of second-order systems.
Let $U$ and $U_0$ be open bounded subsets of $\mathbb{R}^n$ with $0\in U_0$. Consider the system
\begin{align}\label{a1.3}
\ddot{x}=\varepsilon F(\omega t,x,\dot{x},\varepsilon),
\end{align}
where $F:\mathbb{T}^d\times U\times U_0\times[0,\varepsilon_0)\rightarrow \mathbb{R}^n$ is continuous and real analytic on  $\mathbb{T}^d\times U\times U_0$ for each $\varepsilon\in[0,\varepsilon_0)$. Furthermore, $F(\theta,x,y,\varepsilon)$, $\partial_x F(\theta,x,y,\varepsilon)$ and
$\partial_y F(\theta,x,y,\varepsilon)$ are Lipschitz continuous with respect to $\varepsilon$ on $[0,\varepsilon_0)$ uniformly for $(\theta,x,y)\in\mathbb{T}^d\times U\times U_0$.
Let
$$\overline{F_0}(x)=\lim\limits_{T\rightarrow\infty}\frac{1}{T}\int_0^TF(\omega t,x,0,0)\mathrm{d}t,\ \ \ \ x\in U.$$
We have the following result.
\begin{theorem}\label{theorem1}
Assume that $\omega$ satisfies the Diophantine condition \eqref{DC}, and that
there exists a point $x_0\in U$ such that $\overline{F_0}(x_0)=0$ and the eigenvalues of $\partial_x\overline{F_0}(x_0)$ are nonzero and pairwise distinct. Then there exist  a constant $0<\varepsilon_1<\varepsilon_0$ and a Cantorian set $\mathcal{E}\subset(0,\varepsilon_1)$ such that for each $\varepsilon\in\mathcal{E}$, the system
\eqref{a1.3} has a quasi-periodic response solution $x^*(t,\varepsilon)$ satisfying $\lim\limits_{\substack{\varepsilon \in \mathcal{E} \\ \varepsilon \rightarrow 0}}x^*(t,\varepsilon)=x_0$ uniformly for $t\in\mathbb{R}$. The relative Lebesgue measure of $\mathcal{E}\cap(0,\varepsilon)$ with respect to $(0,\varepsilon)$ tends to 1 as $\varepsilon\rightarrow0$.
\end{theorem}

When $F$ is independent of $\dot{x}$, a similar result for Hamiltonian systems was obtained by
 Gentile and Gallavotti \cite{GG}. They considered
the real analytic  Hamiltonian
\begin{align}\label{1.4}
 H=\omega\cdot A+\frac{1}{2}A\cdot A+\frac{1}{2}B\cdot B+\varepsilon f(\alpha,\beta),
 \end{align}
where $(\alpha,A)\in\mathbb{T}^r\times\mathbb{R}^r$ and $(\beta,B)\in\mathbb{T}^s\times\mathbb{R}^s$
with $r, s\geq 1$.
The system written in terms of the angle variables $\alpha$ and $\beta$ alone, is
 \begin{align*}
 \ddot{\alpha}=-\varepsilon\partial_{\alpha}f(\alpha,\beta),\ \ \ \  \ddot{\beta}=-\varepsilon\partial_{\beta}f(\alpha,\beta).
 \end{align*}
Let $\overline{f}(\beta)$ be the average of $f(\alpha,\beta)$ with respect to $\alpha$.
It is proved that if $\omega$ satisfies the Diophantine condition \eqref{DC}, and there exists $\beta_0$ such that $\partial_{\beta}\overline{f}(\beta_0)=0$
and the eigenvalues of the matrix $\partial^2_{\beta}\overline{f}(\beta_0)$ are all different from zero and pairwise distinct, then
there exist $\varepsilon_0>0$ and a set $\mathcal{E}\subset(0,\varepsilon_0)$ such that for each $\varepsilon\in \mathcal{E}$,  the Hamiltonian \eqref{1.4} has a quasi-periodic solution with frequency $\omega$ and the relative Lebesgue measure of
$\mathcal{E}\cap(0,\varepsilon)$ with respect to $(0,\varepsilon)$ tends to $1$ as $\varepsilon\rightarrow 0$.

In \cite{CG}, Corsi and Gentile considered the $1$-dimensional system
 \begin{align}\label{1.5}
\ddot{\beta}=-\varepsilon\partial_{\beta}f(\omega t,\beta),\ \ \ \ \ \beta\in \mathbb{T}^1.
 \end{align}
They
proved the existence of response solutions without restrictions on $\overline{f}(\beta)$. Recently,
Ma and Xu \cite{MX} found that in some cases, the system may have two response solutions. For $\beta\in\mathbb{T}^n$,
Corsi and Gentile \cite{CG1} obtained the responses solutions of system \eqref{1.5} under the condition that $f$ is even in $\alpha$.

When the vector field depends on $\dot{x}$, Si and Yi \cite{SY} studied the $1$-dimensional system
\begin{align}\label{1.6}
\ddot{x}=\varepsilon F(\omega t,x,\dot{x}).
\end{align}
They showed the existence of response solutions
under the condition that $F(\omega t,x,0)$ is a polynomial in $x$. To our knowledge, there is currently limited literature on the existence of response solutions for
system \eqref{1.6} in the high-dimensional case.

Now we prove Theorem \ref{theorem1}.
\begin{proof}[Proof of Theorem \ref{theorem1}]
By rescaling the parameter $\varepsilon=\epsilon^2$, system \eqref{a1.3} can be written as
\begin{equation}\label{3.72}
\begin{cases}
\dot{x}=\epsilon y,\\
\dot{y}=\epsilon F(\omega t,x,\epsilon y,\epsilon^2).
\end{cases}
\end{equation}
Let  $f(\omega t,x,y,\epsilon)=(y,F(\omega t,x,\epsilon y,\epsilon^2))^\top$. Clearly, $f$
satisfies the smoothness conditions of Theorem \ref{theorem}.
It follows from $\overline{F_0}(x_0)=0$ that $\overline{f_0}(x_0,0)=0$, where
$f_0(\omega t,x,y)=f(\omega t,x,y,0)=(y,F(\omega t,x,0,0))^\top$.
 A direct calculation shows that the eigenvalues of matrix $\partial_x\overline{f_0}(x_0,0)$ are the square roots of the eigenvalues of $-\partial_x\overline{F_0}(x_0)$.
Since the eigenvalues of $\partial_x\overline{F_0}(x_0)$ are nonzero and distinct from each other, we obtain that the eigenvalues of   $\partial_x\overline{f_0}(x_0,0)$ are also nonzero and pairwise distinct.
Consequently, Theorem \ref{theorem1} can be obtained from Theorem \ref{theorem}.
\end{proof}

\subsection{Response solutions for perturbed homogeneous systems}
Cheng, de la Llave, and Wang \cite{CdW} studied the response solutions for the $n$-dimensional system:
\begin{equation}\label{1.7}
\dot{x}=\phi(x)+h(\omega t,x)+\varepsilon f(\omega t, x),
\end{equation}
where $\phi: \mathbb{R}^n\rightarrow\mathbb{R}^n$ is a homogeneous function of degree $l$ with $l\geq 2$ an integer, i.e., $\phi(\lambda x)=\lambda^l\phi(x)$
for all $\lambda>0$ and $x\in\mathbb{R}^n$, and $h$ vanishes to order $l+1$ with respect to $x$.
It is proved that if $\overline{f}(0)\neq0$, there exists an $x_0\in\mathbb{R}^n$ such that $\phi(x_0)=-\overline{f}(0)$, and that
the real parts of all the eigenvalues of $\partial_x\phi(x_0)$ are nonzero, then for sufficiently small $\varepsilon>0$, system
\eqref{1.7} has a response solution.

By applying Theorem \ref{th1.2}, we can immediately obtain the response solution for the real analytic system \eqref{1.7} in the case where $\partial_x\phi(x_0)$ has purely imaginary eigenvalues. In fact, by scaling $x=\varepsilon^{\frac{1}{l}} y$, system \eqref{1.7} is transformed into
\begin{equation*}
\dot{y}=\epsilon^{l-1} \phi_1(\omega t,y)+\epsilon^lg(\omega t,y,\epsilon),
\end{equation*}
where $\epsilon=\varepsilon^{1/l} $, $\phi_1(\omega t,y)=\phi(y)+f(\omega t,0)$, and
$$g(\omega t,y,\epsilon)=\frac{1}{\epsilon^{l+1}}h(\omega t,\epsilon y)+\frac{1}{\epsilon}(f(\omega t,\epsilon y)-f(\omega t,0)).$$
Since the system is analytic and $h$ vanishes to order $l+1$ with respect to $y$, it follows that $\phi_1$ and $g$ satisfy the conditions of Theorem \ref{th1.2}.
Note that $\overline{\phi_1}(y)=\phi(y)+\overline{f}(0)$. Applying Theorem  \ref{th1.2}, we immediately have the following result.
\begin{theorem}
Assume that $\omega$ satisfies the Diophantine condition \eqref{DC}, and that
there exists a point $x_0$ such that $\phi(x_0)=-\overline{f}(0)$ and the eigenvalues of $\partial_x\phi(x_0)$ are nonzero and pairwise distinct. Then
there exist a constant $\varepsilon_1>0$ and a Cantorian set $\mathcal{E}\subset(0,\varepsilon_1)$ such that for each $\varepsilon\in\mathcal{E}$, system
\eqref{1.7} has a quasi-periodic response solution $x^*(t,\varepsilon)$ satisfying $\lim\limits_{\substack{\varepsilon \in \mathcal{E} \\ \varepsilon \rightarrow 0}}x^*(t,\varepsilon)=0$ uniformly for $t\in\mathbb{R}$. The relative Lebesgue measure of $\mathcal{E}\cap(0,\varepsilon)$ with respect to $(0,\varepsilon)$ tends to $1$ as $\varepsilon\rightarrow0$.
\end{theorem}

\section{Proof of Theorem \ref{th1.2}}\label{sec3}
Choosing $\tilde{c} \geq 1$ sufficiently large such that $a\tilde{c} \geq 2$ and $(b-a)\tilde{c} \geq 2$, and rescaling the parameter via $\varepsilon = \epsilon^{\tilde{c}}$, we rewrite system \eqref{01.5} as
\begin{equation}\label{e2.11}
\dot{x}=\epsilon^c f(\omega t,x)+\epsilon^{c+1}g_1(\omega t,x,\epsilon), \quad \epsilon\in(0,\epsilon_0),
\end{equation}
where $\epsilon_0=\varepsilon_0^{{1}/{\tilde{c}}}$, $c=a\tilde{c}$, and $g_1(\omega t,x,\epsilon)=\epsilon^{(b-a)\tilde{c}-1}g(\omega t,x,\epsilon^{\tilde{c}})$.
One can easily see that the function $g_1$ satisfies the same smoothness conditions as $g$. The purpose of this scaling transformation is to ensure that the system is Lipschitz continuous with respect to the parameters when performing the subsequent coordinate transformation.

The main idea of the proof is that by applying a fundamental averaging lemma for perturbed quasi-periodic systems, system \eqref{e2.11} can be transformed into
\begin{align*}
\dot{z}&=\left(\epsilon^c\partial_xf(x_0)+ \epsilon^{c+1} B(t,\epsilon)\right)z+\epsilon^{c+1} p( t,\epsilon)
+\epsilon^c h(t,z,\epsilon),
\end{align*}
where $B$, $p$, and $h$ are quasi-periodic in $t$ with frequency vector $\omega$,
and $h(t,z,\epsilon)=\mathcal{O} (|z|^2)$ as $z\rightarrow0$.
Subsequently, using the KAM method developed in \cite{Jo}, the system can be further reduced to
\begin{align}\label{01.3}
\dot{z}=\epsilon^c A_{\infty}(\epsilon)z+\epsilon^c h_{\infty}(t,z,\epsilon),
\end{align}
where $A_{\infty}(\epsilon)$ is a constant matrix and $h_{\infty}(t,z,\varepsilon)=\mathcal{O}(|z|^2)$ as $z\rightarrow0$.
Since each transformation preserves $\omega$-quasi-periodicity, the zero solution of \eqref{01.3}  corresponds to the response solution of the original system \eqref{e2.11}.

In the $m$-th step of the KAM iteration, the system is transformed into the form:
\begin{align*}
\dot{z}_m = \left(\epsilon^c  A_m(\epsilon) + \epsilon^{2^m+c} B_m(t, \epsilon) \right) z_m + \epsilon^{2^m + c} p_m(t, \epsilon) + \epsilon^c h_m(t, z_m, \epsilon),
\end{align*}
where $h_m(t, z_m, \epsilon)=\mathcal{O}(|z_m|^2)$ as $z_m\rightarrow0$.
The $(m+1)$-th KAM iteration consists of two transformations. The first step is to reduce $\epsilon^{2^m + c} p_m(t, \epsilon)$, transforming the system into:
\begin{align*}
\dot{y}_m = \left(\epsilon^c A_m^*(\epsilon) + \epsilon^{2^m+c} B_m^*(t, \epsilon) \right) y_m + \epsilon^{2^{m+1} + c} p_m^*(t, \epsilon) + \epsilon^c h_m^*(t, y_m, \epsilon),
\end{align*}
which requires the eigenvalues $\epsilon^c \lambda_m^{(i)}(\epsilon)$ ($i\in\{1,\dots,n\}$) of $\epsilon^c A_m(\epsilon)$ to satisfy the non-resonance condition:
\begin{align*}
\left| \sqrt{-1} \langle k, \omega \rangle - \epsilon^c \lambda_m^{(i)}(\epsilon) \right| \geq \frac{\gamma}{2} |k|^{-\tau_m} e^{-\nu_m |k|}, \quad \forall k \in \mathbb{Z}^d \backslash \{0\}, \quad 1 \leq i \leq n,
\end{align*}
where $\nu_m$ and $\tau_m$ are the iteration parameters. The second step is to reduce $\epsilon^{2^m + c} B_m^*(t, \epsilon)$, which requires the eigenvalues $\epsilon^c \lambda_m^{*,(1)}(\epsilon),\dots,\epsilon^c \lambda_m^{*,(n)}(\epsilon)$ of $\epsilon^c A_m^*(\epsilon)$ to satisfy the non-resonance condition:
\begin{align*}
\left| \sqrt{-1} \langle k, \omega \rangle - \epsilon^c \lambda_m^{*,(i)}(\epsilon) + \epsilon^c \lambda_m^{*,(j)}(\epsilon) \right| \geq \frac{\gamma}{2} |k|^{-\tau_m} e^{-\nu_m |k|},
\end{align*}
for all $k\in\mathbb{Z}^d\backslash\{0\}$ and $1\leq i,j\leq n$.
To estimate the measure of the set of $\epsilon$ that does not satisfy these non-resonance conditions, for sufficiently small $\epsilon_1 > 0$, it is necessary that $\epsilon^c \lambda_m^{(i)}(\epsilon)$ and $\epsilon^c \lambda_m^{*,(i)}(\epsilon) - \epsilon^c \lambda_m^{*,(j)}(\epsilon)$ are Lipschitz from below with respect to $\epsilon$ on $(0, \epsilon_1)$. In \cite{Jo}, to ensure this, non-degeneracy conditions were assumed, i.e., the eigenvalues of $\underline{A}$ mentioned in the introduction satisfy a bi-Lipschitz condition with respect to the parameter.

In this work, we show that a uniform upper bound exists for the Lipschitz constants of
$\lambda_m^{(i)}(\epsilon)$ and $\lambda_m^{*,(i)}(\epsilon) - \lambda_m^{*,(j)}(\epsilon)$  rather than for $\epsilon^c \lambda_m^{(i)}(\epsilon)$ and $\epsilon^c \lambda_m^{*,(i)}(\epsilon) - \epsilon^c \lambda_m^{*,(j)}(\epsilon)$.
Moreover, we prove that $A_m(\epsilon)$ and $A_m^*(\epsilon)$ are $\epsilon$-close to $\partial_x f(x_0)$. Consequently, $\epsilon^c \lambda_m^{(i)}(\epsilon)$ and $\epsilon^c \lambda_m^{*,(i)}(\epsilon) - \epsilon^c \lambda_m^{*,(j)}(\epsilon)$ are Lipschitz from below on the interval $(\epsilon_1^\delta, \epsilon_1)$ (where $\delta > 1$), while the measure of $(0, \epsilon_1^\delta)$ is small relative to $(0, \epsilon_1)$. Therefore, for the systems considered in this paper, we are able to provide a measure estimate for the parameter set without requiring the non-degeneracy conditions similar to those in \cite{Jo}.

\subsection{Fundamental averaging lemma}
We denote by $\|\cdot\|$ the sup norm of a vector or matrix.
For given $\rho,r>0$, $z\in\mathbb{C}^n$, and $A\in\mathbb{R}^{n\times n}$, define
\begin{align*}
&\mathbb{T}^d_{\rho}=\{\theta\in \mathbb{C}^d/2\pi\mathbb{Z}^d:\ \|{\rm Im } \theta\|\leq \rho\},\\
&\mathbf{B}_r(z)=\{x\in\mathbb{C}^n:\ \|x-z\|\leq r\}, \quad \mathbf{B}_r=\mathbf{B}_r(0),\\
&\mathrm{B}_r(A)=\{C\in\mathbb{R}^{n\times n}:\ \|C-A\|\leq r\}.
\end{align*}
For a real analytic function $f:\mathbb{T}^d_{\rho}\rightarrow\mathbb{C}^n$,  we denote the norm $$\|f\|_{\rho}=\sup\limits_{\theta\in\mathbb{T}^d_{\rho}}\|f(\theta)\|.$$

Now fix $\rho, r>0$ such that $f(\cdot,\cdot)$ and $g(\cdot,\cdot,\varepsilon)$ are analytic on $\mathbb{T}^d_{2\rho}\times\mathbf{B}_{2r}(x_0)$ for each $\varepsilon\in(0,\varepsilon_0)$. Furthermore, assume that $g(\theta,x,\cdot)$ and $\partial_x g(\theta,x,\cdot)$ are Lipschitz continuous on $(0,\varepsilon_0)$ uniformly for $(\theta,x)\in \mathbb{T}^d_{2\rho}\times\mathbf{B}_{2r}(x_0)$. Clearly, $g_1(\cdot,\cdot,\epsilon)$ is also analytic on $\mathbb{T}^d_{2\rho}\times\mathbf{B}_{2r}(x_0)$ for each $\epsilon\in(0,\epsilon_0)$. $g_1(\theta,x,\cdot)$ and $\partial_x g_1(\theta,x,\cdot)$ are also Lipschitz continuous on $(0,\epsilon_0)$ uniformly for $(\theta,x)\in \mathbb{T}^d_{2\rho}\times\mathbf{B}_{2r}(x_0)$.
Let $M=\sup\limits_{x\in\mathbf{B}_{2r}(x_0)}\|f(\cdot,x)\|_{2\rho}$.

The following fundamental lemma ensures that system \eqref{e2.11} can be transformed into a perturbed form of the averaged system.
\begin{lemma}\label{lem1}
Under the hypotheses of Theorem \ref{th1.2}, there exist a real analytic quasi-periodic function $u(t,x)$ with frequency vector $\omega$, defined for $(\omega t,x)\in \mathbb{T}^d_{\rho}\times\mathbf{B}_{r}(x_0)$, and a constant $0<\epsilon_1<\epsilon_0$ such that for each $\epsilon \in (0, \epsilon_1)$,   the change of variables
$$x=y+\epsilon^c u(t,y),\ \ \ (\omega t,y)\in \mathbb{T}^d_{\rho}\times\mathbf{B}_{r}(x_0)$$
transforms \eqref{e2.11} into
\begin{align}\label{2.3}
\dot{y}=\epsilon^c\overline{f}(y)+\epsilon^{c+1} g_2( t,y,\epsilon).
\end{align}
Here, $g_2$ is continuous, quasi-periodic in $t$ with frequency $\omega$, and real analytic on $\mathbb{T}^d_{\rho}\times\mathbf{B}_{r}(x_0)$ for each $\epsilon\in(0,\epsilon_1)$. Moreover, $g_2(t,y,\cdot)$ and $\partial_y g_2(t,y,\cdot)$ are Lipschitz continuous on $(0,\epsilon_1)$ uniformly for $(\omega t,y)\in \mathbb{T}^d_{\rho}\times\mathbf{B}_{r}(x_0)$.
\end{lemma}
\begin{proof}
Let  $\widetilde{f}(t,x)=f(\omega t,x)-\overline{f}(x)$.
Consider the equation
$$\partial_t u(t,x)=\widetilde{f}(t,x),\ \ (\omega t,x)\in \mathbb{T}^d_{\rho}\times\mathbf{B}_{2r}(x_0).$$
Using the Fourier expansions
$$u(t,x)=\sum\limits_{k\in \mathbb{Z}^d}u_k(x)e^{\sqrt{-1}\langle k,\omega\rangle t},\ \ \ f(\omega t,x)=\sum\limits_{k\in \mathbb{Z}^d}f_k(x)e^{\sqrt{-1}\langle k,\omega\rangle t},$$
and comparing coefficients for each $k$, we obtain
$$u_k(x)=\frac{f_k(x)}{\sqrt{-1}\langle k,\omega\rangle},\ \ \ \ \ k\in \mathbb{Z}^d\backslash\{0\}.$$
By the Diophantine condition \eqref{DC} and Cauchy's formula,
$$\|u_k(x)\|\leq\frac{1}{\gamma}|k|^\tau\|f_k(x)\|\leq\frac{1}{\gamma}\|f(\cdot,x)\|_{2\rho}e^{-2\rho |k|}|k|^\tau$$
for $k\in\mathbb{Z}^d\backslash\{0\}$ and $x\in\mathbf{B}_{2r}(x_0)$.
Let $u_0(x)=0$.
It follows that
\begin{align}\label{02.4}
\|u(\cdot,x)\|_{\rho}&\leq\sum\limits_{k\in\mathbb{Z}^d\backslash\{0\}}\|u_k(x)\|e^{\rho |k|}\nonumber\\
&\leq \frac{1}{\gamma}\|f(\cdot,x)\|_{2\rho}\sum\limits_{k\in\mathbb{Z}^d\backslash\{0\}}e^{-\rho |k|}|k|^\tau\nonumber\\
&\leq M M_0
\end{align}
for all $x\in\mathbf{B}_{2r}(x_0)$, where $M_0=\frac{1}{\gamma}\varpi(\tau,\rho)$ and
\begin{align}\label{2.5}
\varpi(\tau,\rho)=\sum\limits_{k\in\mathbb{Z}^d\backslash\{0\}}e^{-\rho |k|}|k|^\tau.
\end{align}
Lemma \ref{ALe2} in the appendix ensures that $\varpi(\tau,\rho)$ is well defined.
Moreover, the Cauchy estimate yields that
\begin{align}\label{2.4}
\|\partial_x u(\cdot,x)\|_{\rho}\leq \frac{1}{r}\sup\limits_{x\in\mathbf{B}_{2r}(x_0)}\|u(\cdot,x)\|_{\rho}\leq\frac{M M_0}{r}
\end{align}
for $x\in\mathbf{B}_{r}(x_0)$.
Let $\epsilon_1\leq\left(\frac{r}{2M M_0}\right)^{1/c}$. Then  for $0<\epsilon<\epsilon_1$,
$I_n+\epsilon^c \partial_x u(t,x)$ is invertible, where $I_n$
is the $n\times n$ identity matrix. Applying the change of variables
$$x=y+\epsilon^c u(t,y),\ \ \ (\omega t,y)\in \mathbb{T}^d_{\rho}\times\mathbf{B}_{r}(x_0),$$
we obtain
\begin{align*}
&\left(I_n+\epsilon^c\partial_y u(t,y)\right)\dot{y}+\epsilon^c\partial_t u(t,y)\\
&=\epsilon^c\overline{f}(y)+\epsilon^c\widetilde{f}(t,y)
+\epsilon^c\left(f(\omega t,y+\epsilon^c u(t,y))-f(\omega t,y)\right)\\
&+\epsilon^{c+1} g_1(\omega t,y+\epsilon^c u(t,y),\epsilon).
\end{align*}
Consequently,
\begin{align*}
\dot{y}=\epsilon^c \overline{f}(y)+\epsilon^{c+1} g_2( t,y,\epsilon),
\end{align*}
where
\begin{align*}
g_2( t,y,\epsilon)&=\frac{1}{\epsilon}\left(\left(I_n+\epsilon^c \partial_y u(t,y)\right)^{-1}\overline{f}(y)-\overline{f}(y)\right)\\
&+\frac{1}{\epsilon}\left(I_n+\epsilon^c \partial_y u(t,y)\right)^{-1}\left(f(\omega t,y+\epsilon^c u(t,y))-f(\omega t,y)\right)\\
&+\left(I_n+\epsilon^c \partial_y u(t,y)\right)^{-1}g_1(\omega t,y+\epsilon^c u(t,y),\epsilon).
\end{align*}
By \eqref{2.4}, we have
\begin{align*}
\left(I_n+\epsilon^c \partial_y u(t,y)\right)^{-1}=I_n-\epsilon^c \partial_y u(t,y)+\mathcal{O}(\epsilon^{2c}).
\end{align*}
Since $f$ and $g_1$ are real analytic in $y$, and $g_1$ and $\partial_y g_1$ are Lipschitz continuous in $\epsilon$ uniformly for $(\omega t,x)\in \mathbb{T}^d_{2\rho}\times\mathbf{B}_{2r}(x_0)$, it follows that
$g_2(t,y,\cdot)$ and $\partial_y g_2(t,y,\cdot)$ are  Lipschitz continuous on $(0,\epsilon_1)$ uniformly for $ (\omega t,y)\in \mathbb{T}^d_{\rho}\times\mathbf{B}_{r}(x_0)$.
\end{proof}

Since $\overline{f}(x_0)=0$, by performing Taylor expansion on \eqref{2.3}, we obtain
\begin{align*}
\dot{y}&=\epsilon^c\partial_x\overline{f}(x_0)(y-x_0)+\epsilon^c\hat{f}(y-x_0)+\epsilon^{c+1} g_2( t,x_0,\epsilon)
+\epsilon^{c+1}\partial_x g_2( t,x_0,\epsilon)(y-x_0)\\
&~~+\epsilon^{c+1}\hat{g}_2(t,y-x_0,\epsilon)\\
&=\epsilon^c\left(\partial_x\overline{f}(x_0)+ \epsilon \partial_x g_2( t,x_0,\epsilon)\right)(y-x_0)
+\epsilon^{c+1} g_2( t,x_0,\epsilon)\\
&~~+\epsilon^{c}\left(\hat{f}(y-x_0)+\epsilon\hat{g}_2(t,y-x_0,\epsilon)\right),
\end{align*}
where $\hat{f}(y-x_0)$ and $\hat{g}_2(t,y-x_0,\epsilon)$ consist of quadratic and
higher-order terms in $y-x_0$.
Making the change of variables  $y=z+x_0$, we have
\begin{align}\label{2.6}
\dot{z}&=\epsilon^c\left(A+ \epsilon B(t,\epsilon)\right)z+\epsilon^{c+1} p( t,\epsilon)
+\epsilon^c h(t,z,\epsilon)
\end{align}
for $(\omega t,z,\epsilon)\in \mathbb{T}^d_{\rho}\times\mathbf{B}_{r}\times(0,\epsilon_1)$, where
$A=\partial_x\overline{f}(x_0)$, $B(t,\epsilon)=\partial_x g_2( t,x_0,\epsilon)$,
$p( t,\epsilon)=g_2( t,x_0,\epsilon)$ and $h(t,z,\epsilon)=\hat{f}(z)+\epsilon\hat{g}_2(t,z,\epsilon)$.
Clearly, for each $\epsilon\in(0,\epsilon_1)$, $B(\cdot,\epsilon)$ and $p(\cdot,\epsilon)$ are real analytic on $\mathbb{T}^d_{\rho}$, and $h(\cdot,\cdot,\epsilon)$
is real analytic on $\mathbb{T}^d_{\rho}\times\mathbf{B}_{r}$. Furthermore, the Lipschitz continuity of $g_2$ and $\partial_x g_2$ implies that $B(t,\cdot)$, $p(t,\cdot)$, and $h(t,z,\cdot)$ are  Lipschitz continuous on $(0,\epsilon_1)$ uniformly for $(\omega t,z)\in\mathbb{T}^d_{\rho}\times\mathbf{B}_{r}$.
We denote their Lipschitz constants by
 $\mathcal{L}(B)$, $\mathcal{L}(p)$, and $\mathcal{L}(h)$, respectively.

\subsection{The KAM iterations}
We will take infinite iterations to make the term $\epsilon^{c+1} p( t,\epsilon)$ in \eqref{2.6} approach $0$.
Each iteration consists of two processes: one to reduce  $\epsilon^{c+1} p( t,\epsilon)$ and the other to reduce $\epsilon B(t,\epsilon)$.

Let  $r_0=r$, $\rho_0=\rho$, $0<c_0<\frac{1}{4}$, $1<\kappa<2$, and define the sequences:
$\rho_{m+1}=\rho_m-\frac{\rho_0}{4(m+1)^2}$, $\sigma_m=\rho_m-\frac{\rho_0}{8(m+1)^2}$,
$\nu_m=\frac{c_0\rho_0}{4(m+1)^2}$, and $\tau_m=\tau\kappa^m$ for $m\in\mathbb{N}$.

Set $A_0=A$, $B_0(t,\epsilon)=B(t,\epsilon)$, $p_0( t,\epsilon)=p( t,\epsilon)$ and $h_0(t,z,\epsilon)=h(t,z,\epsilon)$.
Let $C_0$ be the matrix such that $C_0^{-1}A_0C_0=\diag\{\lambda^{(1)},\dots,\lambda^{(n)} \}$, and let $\beta_0=\max\{\|C_0\|, \|C_0^{-1}\| \}$.
Choose $0<\mu<\frac{1}{2}\min\limits_{\substack{1\leq i,j\leq n\\ i\neq j}}\{|\lambda^{(i)}|,\ |\lambda^{(i)}-\lambda^{(j)}|\}$  and set
$\alpha=\frac{2\mu}{(3n-1)\beta_0^2}$.

Assume that after $m$ iterations, system \eqref{2.6} becomes
\begin{align}\label{2.7}
\dot{z}_m&=\epsilon^c\left(A_m+ \epsilon^{2^m} B_m(t,\epsilon)\right)z_m+\epsilon^{2^m+c} p_m( t,\epsilon)
+\epsilon^c h_m(t,z_m,\epsilon),
\end{align}
where $A_m\in\mathrm{{B}}_\alpha(A)$, $B_m(t,\epsilon)$, $p_m( t,\epsilon)$ and $h_m(t,z,\epsilon)$ are continuous and quasi-periodic in $t$ with frequency $\omega$, $B_m(t,\epsilon)$ and $p_m( t,\epsilon)$ are real analytic on $\mathbb{T}^d_{\rho_m}$ and $h_m(t,z_m,\epsilon)$
consists of quadratic and higher-order terms in $z_m$ and real analytic on $\mathbb{T}^d_{\rho_m}\times\mathbf{B}_{r_m}$ for each $\epsilon\in(0,\epsilon_1)$. Moreover, we also assume the average $\overline{B}_m(\epsilon)=0$.

By Lemma \ref{ALe1} in the appendix and Gerschgorin's theorem (see \cite{W}), there exists a constant $\mu^*>0$, which is independent of $m$, such that the eigenvalues $\lambda_m^{(1)},\dots,\lambda_m^{(n)}$ of the matrix $A_m$ satisfy
\begin{align}\label{a2.16}
\mu<|\lambda_m^{(i)}|<\mu^*,\ \ \mu<|\lambda_m^{(i)}-\lambda_m^{(j)}|<\mu^*,
\end{align}
for $1\leq i,j\leq n$ and $i\neq j$.
Suppose that
\begin{align}\label{2.8}
|\sqrt{-1}\langle k,\omega\rangle-\epsilon^c\lambda_m^{(i)}|\geq\frac{\gamma}{2}|k|^{-\tau_m}e^{-\nu_m|k|},\ \ \forall\ k\in\mathbb{Z}^d\backslash\{0\},\ \ 1\leq i\leq n.
\end{align}

To construct a transformation that reduces $\epsilon^{2^m+c} p_m( t,\epsilon)$, we require the following lemma.
\begin{lemma}
Assume that the Diophantine condition \eqref{2.8} holds.  Then for sufficiently small $\epsilon_1>0$, the equation
\begin{align}\label{2.9}
\dot{u}_m=\epsilon^c A_m u_m+\epsilon^{2^m+c} p_m( t,\epsilon)
\end{align}
admits a unique real analytic quasi-periodic solution $u_m(t,\epsilon)$ for $(\omega t, \epsilon)\in\mathbb{T}^d_{\sigma_m}\times(0,\epsilon_1)$ such that
\begin{align}\label{02.10}
\|u_m\|_{\sigma_{m}}\leq \epsilon^{2^m} \|p_m\|_{\rho_m}L_{1,m},
\end{align}
where $L_{1,m}=4\beta_0^2\left(\frac{1}{\mu}+\frac{2\epsilon^c}{\gamma}\varpi(\tau_m,\nu_m)\right)$. Moreover, $u_m(t,\epsilon)$ is Lipschitz continuous in $\epsilon$, with Lipschitz constant
\begin{align}\label{02.11}
\mathcal{L}(u_m)\leq E_{1,m}\epsilon_1^{2^m}\|p_{m}\|_{\rho_m,\epsilon_1}\left(\mathcal{L}(A_m)+1\right)+E_{2,m}\mathcal{L}\left(\epsilon^{2^m}p_{m}\right),
\end{align}
where
$E_{1,m}=\tilde{M}_1\varpi(2\tau_m,\rho_m-2\nu_m-\sigma_m),$ $E_{2,m}=\tilde{M}_2\varpi(\tau_m,\rho_m-\nu_m-\sigma_m)$,
$\tilde{M}_1$ and $\tilde{M}_2$ are positive constants independent of $m$, and $\|p_m\|_{\rho_m, \epsilon_1} = \sup\limits_{\epsilon \in (0, \epsilon_1)} \|p_{m}(\cdot, \epsilon)\|_{\rho_m}$.
\end{lemma}
\begin{proof}
Since $A_m\in\mathrm{{B}}_\alpha(A)$, by Lemma \ref{ALe1} in the appendix there exists an
$n\times n$ matrix $C_m$ such that
$$C_m^{-1}A_mC_m=D_m=\diag\{\lambda_m^{(1)},\dots,\lambda_m^{(n)}\},$$
and $\max\{\|C_m\|, \|C^{-1}_m\| \}\leq 2\beta_0$.
Making the change of variables $u_m=C_my$, equation \eqref{2.9} becomes
\begin{align}\label{a2.17}
\dot{y}=\epsilon^c D_my+\epsilon^{2^m+c} q_m( t,\epsilon),
\end{align}
where $q_m( t,\epsilon)=C_m^{-1}p_m( t,\epsilon)$. Let $y=(y^{(j)})_{1\leq j\leq n}$, $q_m=(q_m^{(j)})_{1\leq j\leq n}$, and express them as Fourier series
$$y^{(j)}(t)=\sum\limits_{k\in \mathbb{Z}^d}y^{(j)}_ke^{\sqrt{-1}\langle k,\omega\rangle t},\ \ \ q_m^{(j)}(t,\epsilon)=\sum\limits_{k\in \mathbb{Z}^d}q_{m,k}^{(j)}e^{\sqrt{-1}\langle k,\omega\rangle t}.$$
Substituting these series into \eqref{a2.17} and comparing the coefficients yields
\begin{align}\label{a2.18}
y^{(j)}_k=\frac{\epsilon^{2^m+c}q_{m,k}^{(j)}}{\sqrt{-1}\langle k,\omega\rangle-\epsilon^c\lambda_{m}^{(j)}}, \ \ \ \ k\in \mathbb{Z}^d.
\end{align}
It follows that
$|y_0^{(j)}|\leq\epsilon^{2^m}\frac{\|q_m^{(j)}\|_{\rho_m}}{\mu}$, and $$|y^{(j)}_k|\leq\epsilon^{2^m+c}\|q_m^{(j)}\|_{\rho_m}\frac{2|k|^{\tau_m}}{\gamma}e^{-(\rho_m-\nu_m)|k|},\ \  k\in\mathbb{Z}^d\backslash\{0\}.$$
Therefore,
\begin{align*}
\|y^{(j)}\|_{\sigma_m}&\leq \sum\limits_{k\in\mathbb{Z}^d}|y^{(j)}_k|e^{\sigma_m|k|}\\
&\leq\epsilon^{2^m}\|q_m^{(j)}\|_{\rho_m}\left(\frac{1}{\mu}+\frac{2\epsilon^c}{\gamma} \sum\limits_{k\in\mathbb{Z}^d\backslash\{0\}}|k|^{\tau_m}e^{-(\rho_m-\nu_m-\sigma_m)|k|}\right).
\end{align*}
Since $\|q_m\|_{\rho_m}\leq \|C_m^{-1}\|\|p_m\|_{\rho_m}$, by Lemma \ref{ALe1} in the appendix, we obtain
\begin{align*}
\|u_m\|_{\sigma_m}\leq \|C_m\|\|y\|_{\sigma_m}\leq \epsilon^{2^m} \|p_m\|_{\rho_m}L_{1,m}.
\end{align*}

It remains to estimate the Lipschitz constant of $u_m$. From \eqref{a2.18}, Lemma \ref{ALe1} and Lemma \ref{ALe4}, we have
\begin{align*}
&\mathcal{L}\left(y_0^{(j)}\right)=\mathcal{L}\left(\frac{\epsilon^{2^m}q_{m,0}^{(j)}}{\lambda_m^{(j)}}\right)\\
&\leq \mathcal{L}\left(\frac{1}{\lambda_m^{(j)}}\right)\epsilon_1^{2^m}|q_{m,0}^{(j)}|_{\epsilon_1}+
\left|\frac{1}{\lambda_m^{(j)}}\right|_{\epsilon_1}\mathcal{L}\left(\epsilon^{2^m}q_{m,0}^{(j)}\right)\\
&\leq\frac{\epsilon_1^{2^m}|q_{m,0}^{(j)}|_{\epsilon_1}}{\mu^2}\mathcal{L}(\lambda_m^{(j)})+\frac{1}{\mu}\mathcal{L}\left(\epsilon^{2^m}q_m^{(j)}\right)\\
&\leq\frac{2\beta_0\kappa_2\epsilon_1^{2^m}\|p_{m}\|_{\rho_m,\epsilon_1}}{\mu^2}\mathcal{L}(A_m)+\frac{1}{\mu}
\left(2\beta_0\mathcal{L}\left(\epsilon^{2^m}p_{m}\right)+\kappa_1\epsilon_1^{2^m}\|p_{m}\|_{\rho_m,\epsilon_1}\mathcal{L}(A_m)\right)\\
&=\frac{2\beta_0\kappa_2+\mu\kappa_1}{\mu^2}\epsilon_1^{2^m}\|p_{m}\|_{\rho_m,\epsilon_1}\mathcal{L}(A_m)+\frac{2\beta_0}{\mu}
\mathcal{L}(\epsilon^{2^m}p_{m}),
\end{align*}
and
\begin{align*}
\mathcal{L}\left(y^{(j)}_k\right)&\leq\frac{4\epsilon_1^{2^m+c}|q_{m,k}^{(j)}|_{\epsilon_1}}{(\gamma|k|^{-\tau_m}e^{-\nu_m|k|})^2}\mathcal{L}(\epsilon^c\lambda_m^{(j)})
+\frac{2}{\gamma|k|^{-\tau_m}e^{-\nu_m|k|}}\mathcal{L}\left(\epsilon^{2^m+c}q_{m,k}^{(j)}\right)\\
&\leq 8\beta_0\gamma^{-2}\epsilon_1^{2^m+c}|k|^{2\tau_m}e^{-(\rho_m-2\nu_m)|k|}\|p_{m}\|_{\rho_m,\epsilon_1}
(\epsilon_1^c\kappa_2\mathcal{L}(A_m)+\mu^*c\epsilon_1^{c-1})\\
&+4\beta_0\gamma^{-1}|k|^{\tau_m}e^{-(\rho_m-\nu_m)|k|}
\left(c\epsilon_1^{2^m+c-1}\|p_{m}\|_{\rho_m,\epsilon_1}+\epsilon_1^c\mathcal{L}\left(\epsilon^{2^m}p_{m}\right)
\right)\\
&+2\gamma^{-1}\kappa_1|k|^{\tau_m}e^{-(\rho_m-\nu_m)|k|}\epsilon_1^{2^m+c}\|p_{m}\|_{\rho_m,\epsilon_1}\mathcal{L}(A_m),
\end{align*}
for $k\in\mathbb{Z}^d\backslash\{0\}$, where $|\cdot|_{\epsilon_1}$ denotes the sup norm on $(0, \epsilon_1)$. For $\epsilon_1<\frac{1}{c}$, we obtain
\begin{align*}
&\mathcal{L}(u_m)\leq2\beta_0\mathcal{L}(y)+\|y\|_{\sigma_m,\epsilon_1}\kappa_1\mathcal{L}(A_m)\\
&\leq2\beta_0\max\limits_{1\leq j\leq n}\sum\limits_{k\in\mathbb{Z}^d}\mathcal{L}(y^{(j)}_k)e^{\sigma_m|k|}+\|y\|_{\sigma_m,\epsilon_1}\kappa_1\mathcal{L}(A_m)\\
&\leq2\beta_0\max\limits_{1\leq j\leq n}\mathcal{L}(y^{(j)}_0)\\
&+ 16\beta_0^2\gamma^{-2}\epsilon_1^{2^m+c}\varpi(2\tau_m,\rho_m-2\nu_m-\sigma_m)\|p_{m}\|_{\rho_m,\epsilon_1}
(\epsilon_1^c\kappa_2\mathcal{L}(A_m)+\mu^*)\\
&+8\beta_0^2\gamma^{-1}\varpi(\tau_m,\rho_m-\nu_m-\sigma_m)
\left(\epsilon_1^{2^m}\|p_{m}\|_{\rho_m,\epsilon_1}+\epsilon_1\mathcal{L}\left(\epsilon^{2^m}p_{m}\right)
\right)\\
&+4\beta_0\gamma^{-1}\kappa_1\varpi(\tau_m,\rho_m-\nu_m-\sigma_m)\epsilon_1^{2^m+c}\|p_{m}\|_{\rho_m,\epsilon_1}\mathcal{L}(A_m)+\|y\|_{\sigma_m,\epsilon_1}\kappa_1\mathcal{L}(A_m)\\
&\leq E_{1,m}\epsilon_1^{2^m}\|p_{m}\|_{\rho_m,\epsilon_1}\left(\mathcal{L}(A_m)+1\right)+E_{2,m}\mathcal{L}\left(\epsilon^{2^m}p_{m}\right).
\end{align*}
Here we use the fact $\varpi(\tau_m,\rho_m-\nu_m-\sigma_m)\leq \varpi(2\tau_m,\rho_m-2\nu_m-\sigma_m)$.
\end{proof}

\begin{lemma}\label{lem3}
Consider system \eqref{2.7}. Assume that
$$\sup\limits_{z_m\in\mathbf{B}_{r_m}}\|\partial_{x}^2h_m(\cdot,z_m,\epsilon)\|_{\rho_m}\leq K_m$$
for some constant $K_m>0$
and the Diophantine condition \eqref{2.8} holds. Let $u_m$ be the unique real analytic quasi-periodic solution of \eqref{2.9}.
Then the change of variables $z_m=y_m+u_m$
transforms system \eqref{2.7} into
\begin{align}\label{2.11}
\dot{y}_m&=\epsilon^c\left(A_m^*+ \epsilon^{2^m} B_m^*(t,\epsilon)\right)y_m+\epsilon^{2^{m+1}+c} p_m^*(t,\epsilon)
+\epsilon^c h_m^*(t,y_m,\epsilon),
\end{align}
where the average $ \overline{B}_m^*(\epsilon)=0$, $A_m^*$, $\epsilon^{2^m}B_m^*$, $\epsilon^{2^{m+1}}p_m^*$ and $ h_m^*(t,y_m,\epsilon)$ are Lipschitz continuous with respect to $\epsilon$, and the following estimates hold.
\begin{enumerate}
\item[{\rm(1)}]$\|A_m^*\|\leq\|A_m\|+\epsilon^{2^m}K_mL_{1,m}\|p_m\|_{\rho_m}$.
\item[{\rm(2)}] $\|B_m^*\|_{\sigma_m}\leq\|B_m\|_{\rho_m}+2K_mL_{1,m}\|p_m\|_{\rho_m}$.
\item[{\rm(3)}] $\|p_m^*\|_{\sigma_m}\leq \frac{K_mL_{1,m}^2}{2}\|p_m\|_{\rho_m}^2+L_{1,m}\|B_m\|_{\rho_m}\|p_m\|_{\rho_m}$.
\item[{\rm(4)}] $\|\partial_x^2h_m^*(\cdot,y_m,\epsilon)\|_{\sigma_m}\leq K_m$  $\forall\ y_m\in \mathbf{B}_{r_m^*}$,
 where $r_m^*=r_m-\|u_m\|_{\sigma_m}$.
 \item[{\rm(5)}] $\mathcal{L}(A_m^*)\leq\mathcal{L}(A_m)+\eta_m$,
where \begin{align*}
\eta_m&=E_{1,m}\epsilon_1^{2^m}\|p_{m}\|_{\rho_m,\epsilon_1}
\left(\mathcal{L}(A_m)+1\right)+E_{2,m}\mathcal{L}\left(\epsilon^{2^m}p_{m}\right)\\
&+\Delta_1\left(\frac{\|u_m\|_{\sigma_m,\epsilon_1}}{r_m}\right)\mathcal{L}(h_m)\frac{\|u_m\|_{\sigma_m,\epsilon_1}}{r_m},
\end{align*}
with the function $\Delta_1$ defined in Lemma \ref{ALe3}.
\item[{\rm(6)}] $\mathcal{L}\left(\epsilon^{2^m}B_m^*\right)\leq\mathcal{L}\left(\epsilon^{2^m}B_m\right)+2\eta_m$.
 \item[{\rm(7)}] \begin{align*}
\mathcal{L}\left(\epsilon^{2^{m+1}}p_m^*\right)&\leq \mathcal{L}\left(\epsilon^{2^m}B_m\right)\|u_m\|_{\sigma_m,\epsilon_1}+\left(\epsilon_1^{2^m}\|B_m\|_{\rho_m,\epsilon_1}
+K_m\|u_m\|_{\sigma_m,\epsilon_1}\right)\mathcal{L}(u_m)\\
&+\Delta_2\left(\frac{\|u_m\|_{\sigma_m,\epsilon_1}}{r_m}\right)\mathcal{L}(h_m)\frac{\|u_m\|^2_{\sigma_m,\epsilon_1}}{r^2_m},
\end{align*}
where the function $\Delta_2$ is given in Lemma \ref{ALe3}.
 \item[{\rm(8)}] $\mathcal{L}\left(h_m^*\right)\leq 3K_mr_m\mathcal{L}(u_m)+2\mathcal{L}(h_m)+ \Delta_1\left(\frac{\|u_m\|_{\sigma_m,\epsilon_1}}{r_m}\right)\mathcal{L}(h_m)\|u_m\|_{\sigma_m,\epsilon_1}.$
\end{enumerate}
\end{lemma}
\begin{proof}
Applying the change of variables $z_m=y_m+u_m(t,\epsilon)$ in system \eqref{2.7}, we obtain
\begin{align*}
\dot{y}_m+\dot{u}_m&=\epsilon^c\left(A_m+ \epsilon^{2^m} B_m(t,\epsilon)\right)(y_m+u_m)+\epsilon^{2^m+c} p_m(t,\epsilon)\\
&+\epsilon^c h_m(t,y_m+u_m,\epsilon).
\end{align*}
Using \eqref{2.9},
\begin{align}\label{2.12}
\dot{y}_m&=\epsilon^c\left(A_m+ \epsilon^{2^m} B_m(t,\epsilon)+\partial_xh_m(t,u_m,\epsilon)\right)y_m+\epsilon^{2^m+c} B_m(t,\epsilon)u_m\nonumber\\
&+\epsilon^c h_m(t,u_m,\epsilon)+\epsilon^c h_m(t,y_m+u_m,\epsilon)-\epsilon^c h_m(t,u_m,\epsilon)-\epsilon^c\partial_xh_m(t,u_m,\epsilon)y_m.
\end{align}
Define
\begin{align}
&A_m^*=A_m+\overline{\partial_xh_m}(u_m,\epsilon),\label{02.15}\\
&B_m^*(t,\epsilon)=B_m(t,\epsilon)+\frac{1}{\epsilon^{2^m}}\widetilde{\partial_xh}_m(t,u_m,\epsilon),\label{02.16}\\
&p_m^*(t,\epsilon)=\frac{1}{\epsilon^{2^m}}B_m(t,\epsilon)u_m+\frac{1}{\epsilon^{2^{m+1}}}h_m(t,u_m,\epsilon),\label{02.17}\\
&h_m^*(t,y_m,\epsilon)=h_m(t,y_m+u_m,\epsilon)- h_m(t,u_m,\epsilon)-\partial_xh_m(t,u_m,\epsilon)y_m.\label{02.18}
\end{align}
Then the system \eqref{2.12} takes the form of \eqref{2.11}. We now establish the bounds for each term. Using \eqref{02.10}, we obtain
\begin{align*}
\|A_m^*\|&\leq\|A_m\|+\|\partial_xh_m(\cdot,u_m,\epsilon)\|_{\sigma_m}\leq\|A_m\|+K_m\|u_m\|_{\sigma_m}\\
&\leq\|A_m\|+\epsilon^{2^m}K_m L_{1,m}\|p_m\|_{\rho_m},
\end{align*}
\begin{align*}
\|B_m^*\|_{\sigma_m}\leq\|B_m\|_{\rho_m}+\frac{2}{\epsilon^{2^m}}K_m\|u_m\|_{\sigma_m}\leq\|B_m\|_{\rho_m}+2K_mL_{1,m}\|p_m\|_{\rho_m},
\end{align*}
\begin{align*}
\|p_m^*\|_{\sigma_m}&\leq\frac{1}{\epsilon^{2^m}}\|B_m\|_{\rho_m}\|u_m\|_{\sigma_m}
+\frac{1}{\epsilon^{2^{m+1}}}\frac{K_m}{2}\|u_m\|^2_{\sigma_m}\\
&\leq L_{1,m}\|B_m\|_{\rho_m}\|p_m\|_{\rho_m}+\frac{K_mL_{1,m}^2}{2}\|p_m\|_{\rho_m}^2.
\end{align*}
Since $y_m\in \mathbf{B}_{r_m^*}$, $y_m+u_m\in \mathbf{B}_{r_m}$. Thus,
\begin{align*}
\|\partial_x^2h_m^*(\cdot,y_m,\epsilon)\|_{\rho_m}\leq\sup\limits_{z_m\in\mathbf{B}_{r_m}}\|\partial_{x}^2h_m(\cdot,z_m,\epsilon)\|_{\rho_m}\leq K_m.
\end{align*}
We now estimate the Lipschitz constants. From \eqref{02.15}, \eqref{02.11} and Lemma \ref{ALe3} in the appendix, we have
\begin{align*}
\mathcal{L}(A_m^*)&\leq \mathcal{L}\left( A_m\right)+\mathcal{L}\left( \overline{\partial_xh_m}(u_m,\epsilon)\right)\\
&\leq\mathcal{L}(A_m)+K_m\mathcal{L}(u_m)+\sup\limits_{\|x\|\leq \|u_m\|_{\sigma_m,\epsilon_1}}\mathcal{L}\left( \overline{\partial_xh_m}(x,\epsilon)\right)\\
&\leq\mathcal{L}(A_m)+K_mE_{1,m}\epsilon_1^{2^m}\|p_{m}\|_{\rho_m,\epsilon_1}\left(\mathcal{L}(A_m)+1\right)+K_mE_{2,m}\mathcal{L}\left(\epsilon^{2^m}p_{m}\right)\\
&~~+\Delta_1\left(\frac{\|u_m\|_{\sigma_m,\epsilon_1}}{r_m}\right)\mathcal{L}(h_m)\frac{\|u_m\|_{\sigma_m,\epsilon_1}}{r_m}.
\end{align*}
By \eqref{02.16}, \eqref{02.11} and Lemma \ref{ALe3}, we have
\begin{align*}
\mathcal{L}\left(\epsilon^{2^m}B_m^*\right)&\leq\mathcal{L}\left(\epsilon^{2^m}B_m\right)
+\mathcal{L}\left(\widetilde{\partial_xh}_m(t,u_m,\epsilon)\right)\\
&\leq\mathcal{L}\left(\epsilon^{2^m}B_m\right)+2K_m\mathcal{L}(u_m)+\sup\limits_{\|x\|\leq \|u_m\|_{\sigma_m,\epsilon_1}}\mathcal{L}\left( \widetilde{\partial_xh}_m(t,x,\epsilon)\right)\\
&\leq\mathcal{L}\left(\epsilon^{2^m}B_m\right)+2K_mE_{1,m}\epsilon_1^{2^m}\|p_{m}\|_{\rho_m,\epsilon_1}\left(\mathcal{L}(A_m)+1\right)\\
&+2K_mE_{2,m}\mathcal{L}\left(\epsilon^{2^m}p_{m}\right)
+2\Delta_1\left(\frac{\|u_m\|_{\sigma_m,\epsilon_1}}{r_m}\right)\mathcal{L}(h_m)\frac{\|u_m\|_{\sigma_m,\epsilon_1}}{r_m}.
\end{align*}
From \eqref{02.17} and Lemma \ref{ALe3}, we have
\begin{align*}
&\mathcal{L}\left(\epsilon^{2^{m+1}}p_m^*\right)\\
&\leq\mathcal{L}\left(\epsilon^{2^m}B_m(t,\epsilon)u_m\right)
+\mathcal{L}\left(h_m(t,u_m,\epsilon)\right)\\
&\leq\mathcal{L}\left(\epsilon^{2^m}B_m\right)\|u_m\|_{\sigma_m,\epsilon_1}+\epsilon_1^{2^m}\|B_m\|_{\rho_m,\epsilon_1}
\mathcal{L}(u_m)\\
&+\sup\limits_{\|x\|\leq \|u_m\|_{\sigma_m,\epsilon_1}}\|\partial_{x}h_m(\cdot,x,\cdot)\|_{\rho_m,\epsilon_1}\mathcal{L}(u_m)+\sup\limits_{\|x\|\leq \|u_m\|_{\sigma_m,\epsilon_1}}\mathcal{L}(h_m(t,x,\epsilon))\\
&\leq \mathcal{L}\left(\epsilon^{2^m}B_m\right)\|u_m\|_{\sigma_m,\epsilon_1}+\left(\epsilon_1^{2^m}\|B_m\|_{\rho_m,\epsilon_1}+K_m\|u_m\|_{\sigma_m,\epsilon_1}\right)\mathcal{L}(u_m)\\
&+\Delta_2\left(\frac{\|u_m\|_{\sigma_m,\epsilon_1}}{r_m}\right)\mathcal{L}(h_m)\frac{\|u_m\|^2_{\sigma_m,\epsilon_1}}{r^2_m}.
\end{align*}
Using \eqref{02.18} and Lemma \ref{ALe3}, we have
\begin{align*}
\mathcal{L}\left(h_m^*\right)&\leq\mathcal{L}\left(h_m(t,y_m+u_m,\epsilon)\right)+
\mathcal{L}\left(h_m(t,u_m,\epsilon)\right)+\mathcal{L}\left(\partial_xh_m(t,u_m,\epsilon)y_m\right)\\
&\leq \sup\limits_{x\in\mathbf{B}_{r_m}}\|\partial_{x}h_m(\cdot,x,\cdot)\|_{\rho_m,\epsilon_1}\mathcal{L}(u_m)+\mathcal{L}(h_m)\\
&+\sup\limits_{\|x\|\leq \|u_m\|_{\sigma_m,\epsilon_1}}\|\partial_{x}h_m(\cdot,x,\cdot)\|_{\rho_m,\epsilon_1}\mathcal{L}(u_m)+\mathcal{L}(h_m)\\
&+K_m r_m^*\mathcal{L}(u_m)+r_m^*\sup\limits_{\|x\|\leq \|u_m\|_{\sigma_m,\epsilon_1}}\mathcal{L}\left(\partial_xh_m(t,x,\epsilon)\right)\\
&\leq 3K_mr_m\mathcal{L}(u_m)+2\mathcal{L}(h_m)+ \Delta_1\left(\frac{\|u_m\|_{\sigma_m,\epsilon_1}}{r_m}\right)\mathcal{L}(h_m)\|u_m\|_{\sigma_m,\epsilon_1}.
\end{align*}
\end{proof}

We now transform the term $\epsilon^{2^m} B_m^*(t,\epsilon)$ in system \eqref{2.11} into a higher-order term of order $\epsilon^{2^{m+1}}$. Consider the equation
\begin{align}\label{2.14}
\dot{S}_m=\epsilon^c A_m^* S_m-\epsilon^c S_mA_m^*+ \epsilon^{2^m} B_m^*(t,\epsilon),
\end{align}
where $A_m^*$ and $B_m^*(t,\epsilon)$ are given in Lemma \ref{lem3}. Assume $A_m^*\in\mathrm{{B}}_\alpha(A)$.
Let $\lambda^{*,(1)}_{m},\dots,\lambda^{*,(n)}_{m}$ be the eigenvalues
of $A_m^*$. By Lemma \ref{ALe1} and Gerschgorin's theorem, they are also pairwise distinct and satisfy
\begin{align*}
\mu<|\lambda^{*,(i)}_{m}|<\mu^*,\ \ \mu<|\lambda^{*,(i)}_{m}-\lambda^{*,(j)}_{m}|<\mu^*,
\end{align*}
for $1\leq i,j\leq n$ and $i\neq j$.
Assume that
\begin{align}\label{2.15}
|\sqrt{-1}\langle k,\omega\rangle-\epsilon^c\lambda^{*,(i)}_{m}+\epsilon^c\lambda^{*,(j)}_{m}|\geq\frac{\gamma}{2}|k|^{-\tau_m}e^{-\nu_m|k|},
\end{align}
for all $k\in\mathbb{Z}^d\backslash\{0\}$ and $1\leq i,j\leq n$. The following lemma holds.
\begin{lemma}\label{lem4}
Assume  that the Diophantine condition \eqref{2.15} holds. Then system \eqref{2.14} admits a unique real analytic $\omega$ quasi-periodic solution $S_m(t,\epsilon)$ defined on $\mathbb{T}^d_{\rho_{m+1}}$ such that
\begin{align}\label{2.16}
\|S_m\|_{\rho_{m+1}}\leq \epsilon^{2^m}\|B_m^*\|_{\sigma_m}{L}_{2,m},
\end{align}
where $L_{2,m}=\frac{32\beta_0^4}{\gamma}\varpi(\tau_m,\sigma_m-\rho_{m+1}-\nu_m)$.  Moreover, $S_m(t,\epsilon)$ is Lipschitz continuous in $\epsilon$ and its Lipschitz constant satisfies
\begin{align}\label{002.22}
\mathcal{L}(S_m)\leq E_{3,m}\epsilon_1^{2^m}\|B_m^*\|_{\sigma_m}\left(\mathcal{L}(A_m^*)+1\right)+E_{4,m}\mathcal{L}\left(\epsilon^{2^m}B_m^*\right),
\end{align}
where
$E_{3,m}=\tilde{M}_3\varpi(2\tau_m,\sigma_m-\rho_{m+1}-2\nu_m),$ and $E_{4,m}=\tilde{M}_4\varpi(\tau_m,\sigma_m-\rho_{m+1}-\nu_m)$ with
$\tilde{M}_3$ and $\tilde{M}_4$ being positive constants independent of $m$.
\end{lemma}
\begin{proof}
Since  $A_m^*\in\mathrm{{B}}_\alpha(A)$, by Lemma \ref{ALe1} there exists
$n\times n$ matrix $C_m^*$ such that
$$(C^*_m)^{-1}A_m^*C_m^*=D_m^*=\diag\{\lambda^{*,(1)}_m,\dots,\lambda^{*,(n)}_m\},$$
and $\max\{\|C_m^*\|, \|(C_m^*)^{-1}\| \}\leq 2\beta_0$.
Let $\Gamma=(\gamma^{(i,j)})_{1\leq i,j\leq n}=(C^*_m)^{-1}S_mC_m^*$ and $R=(r^{(i,j)})_{1\leq i,j\leq n}=(C^*_m)^{-1}B_m^*C_m^*$.
Then equation \eqref{2.14} becomes
\begin{align}\label{a2.29}
\dot{\Gamma}=\epsilon^c D_m^* \Gamma-\epsilon^c\Gamma D_m^*+ \epsilon^{2^m} R(t,\epsilon).
\end{align}
Expanding $\gamma^{(i,j)}(t,\epsilon)$ and $r^{(i,j)}(t,\epsilon)$ in Fourier series:
$$\gamma^{(i,j)}(t,\epsilon)=\sum\limits_{k\in\mathbb{Z}^d\backslash\{0\}}\gamma^{(i,j)}_ke^{\sqrt{-1}\langle k,\omega\rangle t},
 \quad r^{(i,j)}(t,\epsilon)=\sum\limits_{k\in\mathbb{Z}^d\backslash\{0\}}r^{(i,j)}_ke^{\sqrt{-1}\langle k,\omega\rangle t}$$
for $1\leq i,j\leq n$.
Substituting into \eqref{a2.29} and comparing coefficients yields
\begin{align}\label{02.22}
\gamma^{(i,j)}_k=\frac{\epsilon^{2^m}r^{(i,j)}_k}{\sqrt{-1}\langle k,\omega\rangle-\epsilon^c(\lambda^{*,(i)}_m-\lambda^{*,(j)}_m)}, \ \ \ \ \ k\in\mathbb{Z}^d\backslash\{0\}.
\end{align}
By the Cauchy estimate and \eqref{2.15},
\begin{align*}
|\gamma^{(i,j)}_k|\leq\epsilon^{2^m}\|r^{(i,j)}\|_{\sigma_m}\frac{2|k|^{\tau_m}}{\gamma}e^{-(\sigma_m-\nu_m)|k|}, \ \ \ \ \ k\in\mathbb{Z}^d\backslash\{0\}.
\end{align*}
Therefore, for $\omega t\in\mathbb{T}^d_{\rho_{m+1}}$,
\begin{align*}
|\gamma^{(i,j)}(t,\epsilon)|&\leq\sum\limits_{k\in\mathbb{Z}^d\backslash\{0\}}|\gamma^{(i,j)}_k||e^{\sqrt{-1}\langle k,\omega\rangle t}|\\
&\leq\epsilon^{2^m}\sum\limits_{k\in\mathbb{Z}^d\backslash\{0\}}\|r^{(i,j)}\|_{\sigma_m}\frac{2|k|^{\tau_m}}{\gamma}e^{-(\sigma_m-\nu_m-\rho_{m+1})|k|}\\
&\leq \epsilon^{2^m}\|r^{(i,j)}\|_{\sigma_m}\frac{2\varpi(\tau_m,\sigma_m-\nu_m-\rho_{m+1})}{\gamma}.
\end{align*}
The estimate \eqref{2.16} then follows from Lemma \ref{ALe1}.

It remains to  estimate the Lipschitz constants of $S_m$. From \eqref{02.22}, Lemma \ref{ALe1}, and Lemma \ref{ALe4}, we have
\begin{align*}
\mathcal{L}\left(\gamma^{(i,j)}_k\right)\leq\frac{4\epsilon_1^{2^m}|r^{(i,j)}_k|_{\epsilon_1}}{(\gamma|k|^{-\tau_m}e^{-\nu_m|k|})^2}
\mathcal{L}(\epsilon^c(\lambda^{*,(i)}_m-\lambda^{*,(j)}_m))
+\frac{2}{\gamma|k|^{-\tau_m}e^{-\nu_m|k|}}\mathcal{L}\left(\epsilon^{2^m}r^{(i,j)}_k\right).
\end{align*}
Applying Lemmas \ref{ALe1} and \ref{ALe4} yields
\begin{align*}
\mathcal{L}(\Gamma)
&\leq \epsilon_1^{2^m}\frac{4}{\gamma^2}\|R\|_{\sigma_m,\epsilon_1}\varpi(2\tau_m,\sigma_m-\rho_{m+1}-2\nu_m)\left(\mu^*c\epsilon_1^{c-1}+2\epsilon_1^c\kappa_2\mathcal{L}(A_m^*)\right)\\
&+\frac{2}{\gamma}\varpi(\tau_m,\sigma_m-\rho_{m+1}-\nu_m)\mathcal{L}\left(\epsilon^{2^m}R\right)\\
&\leq \epsilon_1^{2^m}\frac{16\beta_0^2}{\gamma^2}\|B_m^*\|_{\sigma_m,\epsilon_1}\varpi(2\tau_m,\sigma_m-\rho_{m+1}-2\nu_m)\left(\mu^*c\epsilon_1^{c-1}+2\epsilon_1^c\kappa_2\mathcal{L}(A_m^*)\right)\\
&+\frac{2}{\gamma}\varpi(\tau_m,\sigma_m-\rho_{m+1}-\nu_m)
\left(4\beta_0\kappa_1\epsilon_1^{2^m}\|B_m^*\|_{\sigma_m,\epsilon_1}\mathcal{L}(A_m^*)+4\beta_0^2\mathcal{L}(\epsilon^{2^m}B_m^*)\right).
\end{align*}
Combining \eqref{2.16} and noticing that
$$\varpi(2\tau_m,\sigma_m-\rho_{m+1}-2\nu_m)\leq\varpi(\tau_m,\sigma_m-\rho_{m+1}-\nu_m),$$
we obtain
\begin{align*}
\mathcal{L}(S_m)&\leq\mathcal{L}\left(C^*_m\right)\|\Gamma\|_{\rho_{m+1},\epsilon_1}
\left(\sup\limits_{\epsilon\in(0,\epsilon_1)}\|(C^*_m)^{-1}\|\right)\\
&+\left(\sup\limits_{\epsilon\in(0,\epsilon_1)}\|C^*_m\|\right)\mathcal{L}\left(\Gamma\right)\left(\sup\limits_{\epsilon\in(0,\epsilon_1)}
\|(C^*_m)^{-1}\|\right)\\
&~~~+\left(\sup\limits_{\epsilon\in(0,\epsilon_1)}\|C^*_m\|\right)\|\Gamma\|_{\rho_{m+1},\epsilon_1}\mathcal{L}\left((C^*_m)^{-1}\right)\\
&\leq \kappa_1\mathcal{L}(A_m^*)4\beta_0^2\|S_m\|_{\rho_{m+1}}2\beta_0
+4\beta_0^2\mathcal{L}\left(\Gamma\right)+2\beta_04\beta_0^2\|S_m\|_{\rho_{m+1},\epsilon_1}\kappa_1\mathcal{L}(A_m^*)\\
&\leq E_{3,m}\epsilon_1^{2^m}\|B_m^*\|_{\sigma_m,\epsilon_1}\left(\mathcal{L}(A_m^*)+1\right)+E_{4,m}\mathcal{L}\left(\epsilon^{2^m}B_m^*\right).
\end{align*}
\end{proof}

\begin{lemma}\label{lem5}
Assume that the Diophantine condition \eqref{2.15} holds. Then the change of variables $y_m=(I_n+\epsilon^c S_m(t,\epsilon))z_{m+1}$ makes system \eqref{2.11}
become
\begin{align}\label{2.18}
\dot{z}_{m+1}&=\epsilon^c\left(A_{m+1}+ \epsilon^{2^{m+1}} B_{m+1}(t,\epsilon)\right)z_{m+1}+\epsilon^{2^{m+1}+c} p_{m+1}( t,\epsilon)\nonumber\\
&+\epsilon^c h_{m+1}(t,z_{m+1},\epsilon),
\end{align}
where the average $\overline{B}_{m+1}=0$ and the following estimates hold.
\begin{enumerate}
\item[{\rm(1)}]$\|A_{m+1}\|\leq\|A_m^*\|+\epsilon^{2^m+c}\frac{\|S_m\|_{\rho_{m+1}}}{1-\epsilon^c\|S_m\|_{\rho_{m+1}}}\|B_m^*\|_{\sigma_m}$.
\item[{\rm(2)}] $\|B_{m+1}\|_{\rho_{m+1}}\leq\frac{2\|S_m\|_{\rho_{m+1}}}{\epsilon^{2^{m}-c}(1-\epsilon^c\|S_m\|_{\rho_{m+1}})}\|B_m^*\|_{\sigma_m}$.
\item[{\rm(3)}] $\|p_{m+1}\|_{\rho_{m+1}}\leq \frac{1}{1-\epsilon^c\|S_m\|_{\rho_{m+1}}}\|p_m^*\|_{\sigma_m}$.
\item[{\rm(4)}] $\|\partial_{x}^2h_{m+1}(\cdot,z_{m+1},\epsilon)\|_{\rho_{m+1}}\leq K_{m+1}$  for all $z_{m+1}\in \mathbf{B}_{r_{m+1}}$,
 where
 \begin{align}\label{2.19}
 K_{m+1}=\frac{(1+\epsilon^c \|S_m\|_{\rho_{m+1}})^2}{1-\epsilon^c\|S_m\|_{\rho_{m+1}}}K_m,\ \ \ \ \ r_{m+1}=\frac{r_m^*}{1+\epsilon^c\|S_m\|_{\rho_{m+1}}}.
 \end{align}
 \item[{\rm(5)}]\begin{align*}
\mathcal{L}(A_{m+1})&\leq\mathcal{L}(A_m^*)+\frac{\epsilon_1^c\|S_m\|_{\rho_{m+1},\epsilon_1}}{1-\epsilon_1^c\|S_m\|_{\rho_{m+1},\epsilon_1}}\mathcal{L}(\epsilon^{2^m}B_m^*)\\
&+\frac{\epsilon_1^{2^m}\|B_m^*\|_{\sigma_m,\epsilon_1}}{(1-\epsilon_1^c\|S_m\|_{\rho_{m+1},\epsilon_1})^2}\left(\|S_m\|_{\rho_{m+1},\epsilon_1}+\epsilon_1^c\mathcal{L}(S_m)\right).
\end{align*}
 \item[{\rm(6)}]\begin{align*}
\mathcal{L}\left(\epsilon^{2^{m+1}}B_{m+1}\right)
&\leq\frac{2\epsilon_1^c\|S_m\|_{\rho_{m+1},\epsilon_1}}{1-\epsilon_1^c\|S_m\|_{\rho_{m+1},\epsilon_1}}\mathcal{L}(\epsilon^{2^m}B_m^*)\\
&+\frac{2\epsilon_1^{2^m}\|B_m^*\|_{\sigma_m,\epsilon_1}}{(1-\epsilon_1^c\|S_m\|_{\rho_{m+1},\epsilon_1})^2}\left(\|S_m\|_{\rho_{m+1},\epsilon_1}
+\epsilon_1^c\mathcal{L}(S_m)\right).
\end{align*}
 \item[{\rm(7)}]\begin{align*}
\mathcal{L}\left(\epsilon^{2^{m+1}}p_{m+1}\right)&\leq\frac{\mathcal{L}(\epsilon^{2^{m+1}}p_m^*)}{1-\epsilon_1^c\|S_m\|_{\rho_{m+1},\epsilon_1}}\\
&+\frac{\epsilon_1^{2^{m+1}}\|p_m^*\|_{\sigma_m,\epsilon_1}}{(1-\epsilon_1^c\|S_m\|_{\rho_{m+1},\epsilon_1})^2}\left(\|S_m\|_{\rho_{m+1},\epsilon_1}+\epsilon_1^c\mathcal{L}(S_m)\right).
\end{align*}
 \item[{\rm(8)}]\begin{align*}
&\mathcal{L}(h_{m+1})
\leq\frac{K_m(r_m^*)^2}{2(1-\epsilon_1^c\|S_m\|_{\rho_{m+1},\epsilon_1})^2}\left(\|S_m\|_{\rho_{m+1},\epsilon_1}+\epsilon_1^c\mathcal{L}(S_m)\right)\\
&+\frac{1}{1-\epsilon_1^c\|S_m\|_{\rho_{m+1},\epsilon_1}}
\left(K_m(r_m^*)^2\|S_m\|_{\rho_{m+1},\epsilon_1}+\epsilon_1^cK_m(r_m^*)^2\mathcal{L}(S_m)+\mathcal{L}(h_m^*)\right).
\end{align*}
\end{enumerate}
\end{lemma}
\begin{proof}
Making the change of variables $y_m=(I_n+\epsilon^c S_m(t,\epsilon))z_{m+1}$ in system \eqref{2.11}, we have
\begin{align*}
(I_n+\epsilon^c S_m)\dot{z}_{m+1}+\epsilon^c\dot{S}_mz_{m+1}&=\epsilon^c\left(A_m^*+ \epsilon^{2^m} B_m^*(t,\epsilon)\right)(I_n+\epsilon^c S_m(t,\epsilon))z_{m+1}\\
&+\epsilon^{2^{m+1}+c} p_m^*(t,\epsilon)
+\epsilon^c h_m^*(t,(I_n+\epsilon^c S_m(t,\epsilon))z_{m+1},\epsilon).
\end{align*}
By \eqref{2.14}, it follows that
\begin{align*}
\dot{z}_{m+1}&=\epsilon^c\left(A_m^*+ \epsilon^{2^m+c} (I_n+\epsilon^c S_m)^{-1}B_m^*S_m\right)z_{m+1}
+\epsilon^{2^{m+1}+c}(I_n+\epsilon^c S_m)^{-1} p_m^*(t,\epsilon)\\
&+\epsilon^c (I_n+\epsilon^c S_m)^{-1}h_m^*(t,(I_n+\epsilon^c S_m)z_{m+1},\epsilon).
\end{align*}
Let $\mathcal{B}(t,\epsilon)=(I_n+\epsilon^c S_m(t,\epsilon))^{-1}B_m^*(t,\epsilon)S_m(t,\epsilon)$ and
\begin{align*}
&A_{m+1}=A_m^*+ \epsilon^{2^m+c}\overline{\mathcal{B}},\\ &B_{m+1}(t,\epsilon)=\frac{1}{\epsilon^{2^{m}-c}}\widetilde{\mathcal{B}}(t,\epsilon),\\
&p_{m+1}(t,\epsilon)=(I_n+\epsilon^c S_m)^{-1} p_m^*(t,\epsilon),\\
&h_{m+1}(t,z_{m+1},\epsilon)=(I_n+\epsilon^c S_m)^{-1}h_m^*(t,(I_n+\epsilon^c S_m(t,\epsilon))z_{m+1},\epsilon).
\end{align*}
Thus, we obtain \eqref{2.18}. We now give the bound of each term. It is straightforward to see that
$$\|(I_n+\epsilon^c S_m)^{-1}\|_{\rho_{m+1}}\leq \sum\limits_{i=0}^\infty(\epsilon^{c})^i\|S_m\|^i_{\rho_{m+1}}\leq\frac{1}{1-\epsilon^c\|S_m\|_{\rho_{m+1}}}.$$
Then,
\begin{align*}
&\|A_{m+1}\|\leq\|A_m^*\|+\epsilon^{2^m+c}\frac{\|S_m\|_{\rho_{m+1}}}{1-\epsilon^c\|S_m\|_{\rho_{m+1}}}\|B_m^*\|_{\sigma_m},\\
&\|B_{m+1}\|_{\rho_{m+1}}\leq\frac{2\|S_m\|_{\rho_{m+1}}}{\epsilon^{2^{m}-c}(1-\epsilon^c\|S_m\|_{\rho_{m+1}})}\|B_m^*\|_{\sigma_m},\\
&\|p_{m+1}\|_{\rho_{m+1}}\leq \frac{1}{1-\epsilon^c\|S_m\|_{\rho_{m+1}}}\|p_m^*\|_{\sigma_m}.
\end{align*}
For $z_{m+1}\in \mathbf{B}_{r_{m+1}}$, we have $(I_n+\epsilon^c S_m)z_{m+1}\in \mathbf{B}_{r_m^*}$ and
\begin{align*}
&\|\partial_{x}^2h_{m+1}(\cdot,z_{m+1},\epsilon)\|_{\rho_{m+1}}\\
&\leq\frac{\|(I_n+\epsilon^c S_m)\|^2_{\rho_{m+1}}}{1-\epsilon^c\|S_m\|_{\rho_{m+1}}}\|\partial_{x}^2h_m^*(\cdot,(I_n+\epsilon^c S_m)z_{m+1},\epsilon)\|_{\rho_{m+1}}\\
&\leq \frac{(1+\epsilon^c \|S_m\|_{\rho_{m+1}})^2}{1-\epsilon^c\|S_m\|_{\rho_{m+1}}}K_m.
\end{align*}

It remains to estimate the Lipschitz constants:
\begin{align*}
\mathcal{L}(A_{m+1})&\leq\mathcal{L}(A_m^*)+\mathcal{L}\left(\epsilon^{2^m+c}\overline{\mathcal{B}}\right)\\
&\leq\mathcal{L}(A_m^*)+\mathcal{L}\left(\epsilon^{2^m+c}(I_n+\epsilon^c S_m)^{-1}B_m^*(t,\epsilon)S_m(t,\epsilon)\right)\\
&\leq\mathcal{L}(A_m^*)+\frac{\epsilon_1^c\|S_m\|_{\rho_{m+1},\epsilon_1}}{1-\epsilon_1^c\|S_m\|_{\rho_{m+1},\epsilon_1}}\mathcal{L}(\epsilon^{2^m}B_m^*)\\
&+\frac{\epsilon_1^{2^m}\|B_m^*\|_{\sigma_m,\epsilon_1}}{(1-\epsilon_1^c\|S_m\|_{\rho_{m+1},\epsilon_1})^2}\left(\|S_m\|_{\rho_{m+1},\epsilon_1}
+\epsilon_1^c\mathcal{L}(S_m)\right),
\end{align*}
\begin{align*}
\mathcal{L}\left(\epsilon^{2^{m+1}}B_{m+1}\right)&=\mathcal{L}\left(\epsilon^{2^{m}+c}\widetilde{\mathcal{B}}\right)\\
&\leq2\mathcal{L}\left(\epsilon^{2^m+c}(I_n+\epsilon^c S_m)^{-1}B_m^*(t,\epsilon)S_m(t,\epsilon)\right)\\
&\leq\frac{2\epsilon_1^c\|S_m\|_{\rho_{m+1},\epsilon_1}}{1-\epsilon_1^c\|S_m\|_{\rho_{m+1},\epsilon_1}}\mathcal{L}(\epsilon^{2^m}B_m^*)\\
&+\frac{2\epsilon_1^{2^m}\|B_m^*\|_{\sigma_m,\epsilon_1}}{(1-\epsilon_1^c\|S_m\|_{\rho_{m+1},\epsilon_1})^2}\left(\|S_m\|_{\rho_{m+1},\epsilon_1}
+\epsilon_1^c\mathcal{L}(S_m)\right),
\end{align*}
\begin{align*}
\mathcal{L}\left(\epsilon^{2^{m+1}}p_{m+1}\right)&=\mathcal{L}\left(\epsilon^{2^{m}+c}(I_n+\epsilon^c S_m)^{-1} p_m^*\right)\\
&\leq\frac{\mathcal{L}(\epsilon^{2^{m+1}}p_m^*)}{1-\epsilon_1^c\|S_m\|_{\rho_{m+1},\epsilon_1}}\\
&+\frac{\epsilon_1^{2^{m+1}}\|p_m^*\|_{\sigma_m,\epsilon_1}}{(1-\epsilon_1^c\|S_m\|_{\rho_{m+1},\epsilon_1})^2}\left(\|S_m\|_{\rho_{m+1},\epsilon_1}
+\epsilon_1^c\mathcal{L}(S_m)\right),
\end{align*}
\begin{align*}
&\mathcal{L}(h_{m+1})=\mathcal{L}\left((I_n+\epsilon^c S_m)^{-1}h_m^*(t,(I_n+\epsilon^c S_m(t,\epsilon))z_{m+1},\epsilon)\right)\\
&\leq\sup\limits_{x\in \mathbf{B}_{r_m^*}}\|h_m^*\|_{\sigma_m,\epsilon_1}\mathcal{L}\left((I_n+\epsilon^c S_m)^{-1}\right)\\
&+\frac{1}{1-\epsilon_1^c\|S_m\|_{\rho_{m+1},\epsilon_1}}\left(\sup\limits_{x\in \mathbf{B}_{r_m^*}}\|\partial_xh_m^*\|_{\sigma_m,\epsilon_1}\mathcal{L}\left(\epsilon^c S_mz_{m+1}\right)+\mathcal{L}(h_m^*)\right)\\
&\leq\frac{K_m(r_m^*)^2}{2(1-\epsilon_1^c\|S_m\|_{\rho_{m+1},\epsilon_1})^2}\left(\|S_m\|_{\rho_{m+1},\epsilon_1}
+\epsilon_1^c\mathcal{L}(S_m)\right)\\
&+\frac{1}{1-\epsilon_1^c\|S_m\|_{\rho_{m+1},\epsilon_1}}
\left(K_m(r_m^*)^2\|S_m\|_{\rho_{m+1},\epsilon_1}+\epsilon_1^cK_m(r_m^*)^2\mathcal{L}(S_m)+\mathcal{L}(h_m^*)\right).
\end{align*}
\end{proof}

\subsection{Convergence of the iterative process}
Now using the fact $\|S_m\|_{\rho_{m+1}}\leq\frac{1}{2}$ for each $\epsilon\in(0,\epsilon_1)$ which will be proved below, and combining Lemma \ref{lem3},  Lemma \ref{lem4} and Lemma \ref{lem5}, we have
\begin{align}\label{2.20}
\|B_{m+1}\|_{\rho_{m+1}}&\leq\frac{2\epsilon^{2^m}\|B_m^*\|_{\sigma_m}{L}_{2,m}}{\epsilon^{2^{m}-c}(1-\epsilon^c\|S_m\|_{\rho_{m+1}})}\|B_m^*\|_{\sigma_m}\nonumber\\
&\leq 4\epsilon^c {L}_{2,m}(\|B_m\|_{\rho_m}+2K_mL_{1,m}\|p_m\|_{\rho_m})^2,
\end{align}
\begin{align}\label{2.21}
\|p_{m+1}\|_{\rho_{m+1}}\leq K_mL_{1,m}^2\|p_m\|_{\rho_m}^2+2L_{1,m}\|B_m\|_{\rho_m}\|p_m\|_{\rho_m},
\end{align}
and
\begin{align}\label{2.22}
&\|A_{m+1}\|\leq \|A_m\|+\epsilon^{2^m+c}\|B_m\|_{\rho_m}+\epsilon^{2^m}(1+2\epsilon^c)K_mL_{1,m}\|p_m\|_{\rho_m}.
\end{align}
It is easy to see that there exists a constant $c_1>0$ such that
\begin{align}\label{02.23}
L_{1,m}\leq c_1{L}_{2,m},\ \ \ \ \ \ \ \forall m\in\mathbb{N}.
\end{align}

By the first formula in \eqref{2.19}, for sufficiently small $\epsilon>0$, we have $K_m\leq2^mK_0$ for $m\in\mathbb{N}$.
Let $\alpha_m=\max\{\|B_{m}\|_{\rho_{m}},\|p_{m}\|_{\rho_{m}}\}$. Then, by \eqref{2.20} and \eqref{2.21}, we obtain
\begin{align*}
\|B_{m+1}\|_{\rho_{m+1}}\leq 4\epsilon^c {L}_{2,m}(1+2^{m+1}K_0c_1{L}_{2,m})^2\alpha_m^2,
\end{align*}
\begin{align*}
\|p_{m+1}\|_{\rho_{m+1}}\leq(2^mK_0c_1^2{L}_{2,m}^2+2c_1{L}_{2,m})\alpha_m^2.
\end{align*}
Choose $\gamma\leq 1$ small enough such that ${L}_{2,m}\geq1$. Let
$$c_2\geq\max\{1,K_0c_1,K_0c_1^2,2c_1\}.$$
Then we obtain
\begin{align*}
\alpha_{m+1}\leq 4^{m+3}c_2^2{L}_{2,m}^3\alpha_m^2.
\end{align*}
By Lemma \ref{ALe2}, a direct calculation shows that there exists a constant $M_1>1$ such that
\begin{align}\label{2.25}
{L}_{2,m}<M_1^{2^m}, \quad \forall m\in\mathbb{N}.
\end{align}
Moreover, it is easy to prove that there exists a constant $M_2>1$ such that $\alpha_m\leq M_2^{2^m}$.
Hence,
\begin{align}\label{2.23}
\|B_{m}\|_{\rho_{m}}\leq M_2^{2^m},\ \ \ \ \ \ \|p_{m}\|_{\rho_{m}}\leq M_2^{2^m}.
\end{align}
Clearly, for $\epsilon_1<\frac{1}{M_2}$ and $\epsilon\in(0,\epsilon_1)$, we have
\begin{align}\label{2.31}
\lim\limits_{m\rightarrow\infty}\epsilon^{2^m}\|B_{m}\|_{\rho_{m}}=0,\ \ \ \ \ \ \lim\limits_{m\rightarrow\infty}\epsilon^{2^m}\|p_{m}\|_{\rho_{m}}=0.
\end{align}
By \eqref{2.22}, we obtain
\begin{align*}
\|A_{m+1}\|\leq \|A_m\|+\epsilon^{2^m}\left(\epsilon^c M_2^{2^m}+2K_0c_1(2M_1M_2)^{2^m}\right)\leq\|A_m\|+\epsilon^{2^m}M_3^{2^m}
\end{align*}
for some constant $M_3>1$.
Then, for $\epsilon<\frac{1}{2M_3}$,
\begin{align}\label{2.32}
\|A_{m}\|\leq \|A_0\|+2\epsilon M_3.
\end{align}
It follows that
$A_m\in\mathrm{B}_{\alpha}(A)$ for $0<\epsilon<\epsilon_1$ with sufficiently small $\epsilon_1$ and $m\in\mathbb{N}$. Similarly, using (1) of Lemma \ref{lem3},
we also have $A_m^*\in\mathrm{B}_{\alpha}(A)$ for $m\in\mathbb{N}$ and sufficiently small $\epsilon_1$.

Now we estimate $\|S_m\|_{\rho_{m+1}}$.
By \eqref{2.16} and Lemma \ref{lem3}, there exists a constant $M_4>1$ such that
\begin{align}\label{2.27}
\|S_m\|_{\rho_{m+1}}\leq \epsilon^{2^m}(\|B_m\|_{\rho_m}+2K_mL_{1,m}\|p_m\|_{\rho_m}){L}_{2,m}\leq \epsilon^{2^m} M_4^{2^m},\ \ \ m\in\mathbb{N}.
\end{align}
Obviously, if $\epsilon>0$ is small enough, $\|S_m\|_{\rho_{m+1}}\leq\frac{1}{2}$ for each $m\in\mathbb{N}$.

Using \eqref{02.10}, \eqref{02.23},  \eqref{2.25} and \eqref{2.23}, we obtain
\begin{align}\label{2.28}
\|u_m\|_{\sigma_{m}}\leq \epsilon^{2^m} \|p_m\|_{\rho_m}L_{1,m}\leq c_1\epsilon^{2^m}(M_2M_1)^{2^m}.
\end{align}
By \eqref{2.19}, we have
 \begin{align*}
 r_{m+1}=\frac{r_m^*}{1+\epsilon^c \|S_m\|_{\rho_{m+1}}}=\frac{r_m-\|u_m\|_{\sigma_m}}{1+\epsilon^c\|S_m\|_{\rho_{m+1}}}.
 \end{align*}
Then,
 \begin{align}
 r_{m+1}\geq r_0\prod\limits_{j=0}^m\frac{1}{1+\epsilon^c \|S_j\|_{\rho_{j+1}}}-\sum\limits_{j=0}^{\infty}\frac{\|u_j\|_{\sigma_j}}{1+\epsilon^c\|S_j\|_{\rho_{j+1}}}.
 \end{align}
By \eqref{2.27} and \eqref{2.28}, for sufficiently small $\epsilon>0$, the sequence $\{r_m\}$ is convergent and
\begin{align}\label{2.35}
\lim\limits_{m\rightarrow\infty}r_m=r_\infty>0.
\end{align}

To ensure the convergence of $h_{m}(t,z_{m},\epsilon)$, we now give an accurate bound of $K_m$.
By the first formula in \eqref{2.19}, for sufficiently small $\epsilon>0$,
 \begin{align*}
 K_{m+1}=\frac{(1+\epsilon^c \|S_m\|_{\rho_{m+1}})^2}{1-\epsilon^c\|S_m\|_{\rho_{m+1}}}K_m\leq (1+2\epsilon^c \|S_m\|_{\rho_{m+1}})^2K_m.
 \end{align*}
It follows from \eqref{2.27} that
$\lim\limits_{m\rightarrow\infty}K_m=K_\infty>0$.
\subsection{Measure estimates on the parameter set}
In order to continue the iteration process, we require that Diophantine conditions \eqref{2.8} and \eqref{2.15} hold for each $m\in\mathbb{N}$.
We now estimate the measure of the set of parameters $\epsilon$ that do not satisfy these Diophantine conditions.

By Lemma \ref{ALe2}, we see that there exists a constant $M_5>1$ such that
\begin{align}
\max\{E_{1,m}, E_{2,m}, E_{3,m}, E_{4,m}\}<M_5^{2^m},\ \ \ \ \ \ \forall m\in\mathbb{N}.
\end{align}
According to  \eqref{2.28} and \eqref{2.35}, for sufficiently small $\epsilon>0$, we have
$$\frac{\|u_m\|_{\sigma_m}}{r_m}\leq\frac{1}{2},\ \ \ \forall m\in\mathbb{N}.$$
Take $M_6\geq\max\left\{M_2M_1M_5, M_2^2M_1,\frac{\Delta_1(\frac{1}{2})}{r_0},\frac{\Delta_2(\frac{1}{2})}{r_0^2}\right\}$ and $\epsilon_1>0$ small enough.
Then \eqref{02.11}, \eqref{2.23}, \eqref{2.28} and Lemma \ref{lem3} yield
 \begin{align}
&\mathcal{L}(u_m)\leq (\epsilon_1M_2M_5)^{2^m}\left(\mathcal{L}(A_m)+1\right)+M_5^{2^m}\mathcal{L}\left(\epsilon^{2^m}p_{m}\right),\\
&\mathcal{L}(A_m^*)\leq\mathcal{L}(A_m)+(\epsilon_1M_6)^{2^m}
\left(\mathcal{L}(A_m)+\mathcal{L}(h_m)+1\right)+M_5^{2^m}\mathcal{L}\left(\epsilon^{2^m}p_{m}\right),\label{2.39}\\
&\mathcal{L}\left(\epsilon^{2^m}B_m^*\right)\leq\mathcal{L}\left(\epsilon^{2^m}B_m\right)+(\epsilon_1M_6)^{2^m}
\left(\mathcal{L}(A_m)+\mathcal{L}(h_m)+1\right)+2M_5^{2^m}\mathcal{L}\left(\epsilon^{2^m}p_{m}\right),\label{2.40}\\
&\mathcal{L}\left(\epsilon^{2^{m+1}}p_m^*\right)\leq (\epsilon_1M_6)^{2^m}\left(\mathcal{L}(A_m)+\mathcal{L}\left(\epsilon^{2^m}B_m\right)
+\mathcal{L}\left(\epsilon^{2^m}p_{m}\right)+\mathcal{L}(h_m)+1\right),\label{2.41}\\
&\mathcal{L}\left(h_m^*\right)\leq 3K_{\infty}r(\epsilon_1M_6)^{2^m}\left(\mathcal{L}(A_m)+1\right)+3\mathcal{L}(h_m)+3K_\infty rM_5^{2^m}\mathcal{L}\left(\epsilon^{2^m}p_{m}\right).\label{2.42}
\end{align}
By (2) and (3) of Lemma \ref{lem3}, \eqref{02.23}, \eqref{2.25} and \eqref{2.23}, we have
\begin{align}\label{2.43}
\|B_m^*\|_{\sigma_m}\leq(1+2c_1K_{\infty})M_6^{2^m}, \ \ \ \|p_m^*\|_{\sigma_m}\leq \left(c_1+\frac{K_\infty c_1^2}{2}\right)M_6^{2^m}.
\end{align}
Then, using \eqref{002.22}, \eqref{2.39} and \eqref{2.40}, we can choose a constant $M_7>1$ and take $\epsilon_1>0$ small enough such that
\begin{align}\label{2.44}
\mathcal{L}(S_m)\leq (\epsilon_1M_7)^{2^m}\left(\mathcal{L}(A_m)+\mathcal{L}\left(h_m\right)+1\right)
+M_7^{2^m}\left(\mathcal{L}\left(\epsilon^{2^m}B_{m}\right)+\mathcal{L}\left(\epsilon^{2^m}p_{m}\right)\right).
\end{align}
According to Lemma \ref{lem5}, \eqref{2.27} and \eqref{2.39}-\eqref{2.44}, by taking a sufficiently large constant $M_8>1$, for sufficiently small $\epsilon_1>0$, we have
\begin{align}\label{2.45}
\mathcal{L}(A_{m+1})&\leq\mathcal{L}(A_m)+M_5^{2^m}\mathcal{L}\left(\epsilon^{2^m}p_{m}\right)\nonumber\\
&+(\epsilon_1M_8)^{2^m}\left(\mathcal{L}(A_m)+\mathcal{L}\left(\epsilon^{2^m}B_{m}\right)+\mathcal{L}\left(\epsilon^{2^m}p_{m}\right)+\mathcal{L}\left(h_m\right)+1\right),
\end{align}
\begin{align}\label{2.46}
\mathcal{L}\left(\epsilon^{2^{m+1}}B_{m+1}\right)
\leq(\epsilon_1M_8)^{2^m}\left(\mathcal{L}(A_m)+
\mathcal{L}\left(\epsilon^{2^m}B_{m}\right)+\mathcal{L}\left(\epsilon^{2^m}p_{m}\right)+\mathcal{L}\left(h_m\right)+1\right),
\end{align}
\begin{align}\label{2.47}
\mathcal{L}\left(\epsilon^{2^{m+1}}p_{m+1}\right)\leq(\epsilon_1M_8)^{2^m}\left(\mathcal{L}(A_m)+\mathcal{L}\left(\epsilon^{2^m}B_{m}\right)+\mathcal{L}\left(\epsilon^{2^m}p_{m}\right)+\mathcal{L}\left(h_m\right)+1\right),
\end{align}
\begin{align}\label{2.48}
\mathcal{L}(h_{m+1})
&\leq(\epsilon_1M_8)^{2^m}\left(\mathcal{L}(A_m)+1\right)
+M_8^{2^m}\left(\mathcal{L}\left(\epsilon^{2^m}B_{m}\right)
+\mathcal{L}\left(\epsilon^{2^m}p_{m}\right)\right)+6\mathcal{L}\left(h_m\right).
\end{align}
Let $\varphi_m=\max\{\mathcal{L}(A_m),\mathcal{L}\left(\epsilon^{2^m}B_{m}\right),\mathcal{L}\left(\epsilon^{2^m}p_{m}\right),\mathcal{L}\left(h_m\right),1\}$.
Then, by \eqref{2.45}-\eqref{2.48} and taking $\epsilon_1M_8<1$, we can choose a constant $M_9>1$ such that
$$\varphi_{m+1}\leq M_9^{2^m}\varphi_m,\ \ \ \ m\in\mathbb{N}.$$
It follows that
\begin{align}\label{02.49}
\varphi_{m}\leq\varphi_0\prod\limits_{j=0}^{m-1}M_9^{2^j}<M_9^{2^{m}}\varphi_0.
\end{align}
Substituting this into \eqref{2.47}, we obtain
\begin{align}\label{02.50}
\mathcal{L}\left(\epsilon^{2^{m+1}}p_{m+1}\right)\leq5\varphi_0(\epsilon_1M_8M_9)^{2^m}.
\end{align}
Then, by \eqref{2.45}, we have
\begin{align}
\mathcal{L}(A_1)\leq\mathcal{L}(A_0)+M_5\|p\|_{\rho,\epsilon_1}+\epsilon_1 M_5\mathcal{L}(p)+5\epsilon_1M_8\varphi_0,
\end{align}
and for $m\geq 1$,
\begin{align*}
\mathcal{L}(A_{m+1})&\leq\mathcal{L}(A_m)
+5 M_5^{2^m}\varphi_0(\epsilon_1M_8M_9)^{2^{m-1}}+5(\epsilon_1M_8)^{2^m}M_9^{2^{m}}\varphi_0\\
&\leq \mathcal{L}(A_m)+5\varphi_0\left(\epsilon_1M_8M_9(M_5^2+\epsilon_1 M_8M_9)\right)^{2^{m-1}}.
\end{align*}
Hence, for sufficiently small $\epsilon_1>0$, there exists a constant $M_{10}>0$ such that
\begin{align*}
\mathcal{L}(A_{m})&\leq M_{10}, \quad \forall m\in\mathbb{N}.
\end{align*}
Let $\delta_1=\left(\frac{2 \kappa_2 M_{10}}{\mu c}\right)^{\frac{1}{c-1}}$ and $\delta_2=\frac{c}{c-1}$. Then, by Lemma \ref{ALe1}, for $1\leq i\leq n$, $m\in\mathbb{N}$, $\epsilon\in(0,\epsilon_1)$ and $\tilde{\epsilon}_1, \tilde{\epsilon}_2\in\left(\delta_1\epsilon^{\delta_2},\epsilon\right)$, we have
\begin{align}\label{2.49}
|\tilde{\epsilon}_1^c\lambda_m^{(i)}(\tilde{\epsilon}_1)-\tilde{\epsilon}_2^c\lambda_m^{(i)}(\tilde{\epsilon}_2)|
&\geq|\tilde{\epsilon}_1^c\lambda_m^{(i)}(\tilde{\epsilon}_1)-\tilde{\epsilon}_2^c\lambda_m^{(i)}(\tilde{\epsilon}_1)|
-|\tilde{\epsilon}_2^c\lambda_m^{(i)}(\tilde{\epsilon}_1)-\tilde{\epsilon}_2^c\lambda_m^{(i)}(\tilde{\epsilon}_2)|\nonumber\\
&\geq\left(c\mu(\delta_1\epsilon^{\delta_2})^{c-1}-\kappa_2 M_{10}\epsilon^c\right)|\tilde{\epsilon}_1-\tilde{\epsilon}_2|\nonumber\\
&\geq \kappa_2 M_{10}\epsilon^c|\tilde{\epsilon}_1-\tilde{\epsilon}_2|.
\end{align}
By choosing $M_{10}$ sufficiently large, it follows from \eqref{2.39}, \eqref{02.49} and \eqref{02.50} that
\begin{align*}
\mathcal{L}(A_m^*)\leq M_{10}, \quad \forall m\in\mathbb{N}.
\end{align*}
In the same way as \eqref{2.49}, for $1\leq i,j\leq n$, $i\neq j$, and $\tilde{\epsilon}_1, \tilde{\epsilon}_2\in\left(\delta_1\epsilon^{\delta_2},\epsilon\right)$ with sufficiently small $\epsilon>0$, we obtain
\begin{align}\label{2.50}
|\tilde{\epsilon}_1^c(\lambda^{*,(i)}_m-\lambda^{*,(j)}_m)(\tilde{\epsilon}_1)-
\tilde{\epsilon}_2^c(\lambda^{*,(i)}_m-\lambda^{*,(j)}_m)(\tilde{\epsilon}_2)|
\geq\kappa_2 M_{10}\epsilon^c|\tilde{\epsilon}_1-\tilde{\epsilon}_2|.
\end{align}
For $\delta>0$, let
\begin{align*}
\mathcal{R}^{(m)}(\delta)=\big\{\varphi\in(0,\delta):~&\exists k'\in\mathbb{Z}^d\backslash\{0\}, {\rm~such~that~}
|\langle k',\omega\rangle-\varphi|<\frac{\gamma}{2}|k'|^{-\tau_m}e^{-\nu_m|k'|}\big\},
\end{align*}
and $\mathcal{R}(\delta)=\bigcup\limits_{m=0}^\infty\mathcal{R}^{(m)}(\delta)$.

The following lemma provides an estimate for the measure of $\mathcal{R}(\delta)$.
\begin{lemma}\label{lem6}
Assume that $\omega$ satisfies the Diophantine condition \eqref{DC}. Then there exist constants $a_1>0$ and $0<a_2<\frac{1}{\tau}$ such that for sufficiently small $\delta>0$, the Lebesgue measure
$\meas(\mathcal{R}(\delta))\leq \delta \exp({-\frac{a_1}{\delta^{a_2}}})$.
\end{lemma}
The proof of Lemma \ref{lem6} is classical. See, for example, Lemma 2.24 of \cite{Jo}. We omit it here.

For $0<\epsilon\leq\epsilon_1$, let
\begin{align*}
&\mathcal{R}_{1,i}^{(m)}(\epsilon)=\big\{\tilde{\epsilon}\in \left(0, \epsilon\right): \tilde{\epsilon}^c\lambda_m^{(i)}(\tilde{\epsilon})\in\mathcal{R}^{(m)}(\mu^*\epsilon^c)\big\},\ \ \ \ \ 1\leq i\leq n,\\
&\mathcal{R}_{2,i,j}^{(m)}(\epsilon)=\big\{\tilde{\epsilon}\in\left(0, \epsilon\right): \tilde{\epsilon}^c(\lambda^{*,(i)}_m(\tilde{\epsilon})-\lambda^{*,(j)}_m(\tilde{\epsilon}))\in\mathcal{R}^{(m)}(\mu^*\epsilon^c)\big\},\ 1\leq i,j\leq n,\ i\neq j,\\
&\mathcal{R}_{1}(\epsilon)=\bigcup\limits_{i=1}^n\bigcup\limits_{m=0}^\infty\mathcal{R}_{1,i}^{(m)}(\epsilon),\quad
\mathcal{R}_{2}(\epsilon)=\bigcup\limits_{\substack{i,j=1\\ i\neq j}}^n\bigcup\limits_{m=0}^\infty\mathcal{R}_{2,i,j}^{(m)}(\epsilon),
\end{align*}
and
$\widetilde{\mathcal{R}}(\epsilon)=\mathcal{R}_{1}(\epsilon)\cup\mathcal{R}_{2}(\epsilon).$

It is easy to see that the parameter set that does not satisfy Diophantine conditions \eqref{2.8} and \eqref{2.15} at each KAM iteration is included in $\widetilde{\mathcal{R}}(\epsilon_1)$. By \eqref{2.49}, \eqref{2.50} and Lemma \ref{lem6}, for $0<\epsilon\leq\epsilon_1$, we have
\begin{align}\label{2.54}
\meas\left(\widetilde{\mathcal{R}}(\epsilon)\cap\left(\delta_1\epsilon^{\delta_2},\epsilon\right)\right)&\leq \frac{n+n^2}{\kappa_2 M_{10}\epsilon^c}\meas\left(\mathcal{R}(\mu^*\epsilon^c)\right)\nonumber\\
&\leq \frac{(n+n^2)\mu^*}{\kappa_2 M_{10}}\exp\left({-\frac{a_1}{(\mu^*\epsilon^c)^{a_2}}}\right).
\end{align}

\subsection{Proof of Theorem \ref{th1.2}}
Now we prove the existence of response solutions for system \eqref{01.5}.

\begin{proof}[Proof of Theorem \ref{th1.2}]
 According to Lemma \ref{lem3} and Lemma \ref{lem5}, and by \eqref{2.31}, for each $\epsilon\in (0,\epsilon_1)\backslash \widetilde{\mathcal{R}}(\epsilon_1)$,
after infinite KAM iterations, system \eqref{2.6} can be transformed into
\begin{align}\label{2.55}
\dot{z}_\infty&=\epsilon^c A_{\infty}z_\infty+\epsilon^c h_{\infty}(t,z_\infty,\epsilon),
\end{align}
where $A_{\infty}=\lim\limits_{m\rightarrow\infty}A_m$ and $h_{\infty}(t,z_\infty,\epsilon)=\lim\limits_{m\rightarrow\infty}h_{m}(t,z_\infty,\epsilon)$, which consists of quadratic and higher-order terms in $z_\infty$.
Moreover, the transformation is given by
\begin{align*}
z=\Phi(t,\epsilon) z_{\infty}+\Psi(t,\epsilon),
\end{align*}
where $\Phi(t,\epsilon)=\lim\limits_{m\rightarrow\infty}\Phi_m(t,\epsilon)$ and $\Psi(t,\epsilon)=\lim\limits_{m\rightarrow\infty}\Psi_m(t,\epsilon)$, with
$\Phi_m(t,\epsilon)=(I_n+\epsilon^c S_0(t,\epsilon))\circ\cdots\circ(I_n+\epsilon^c S_m(t,\epsilon))$ and
$\Psi_m(t,\epsilon)=u_0(t,\epsilon)
+\sum\limits_{j=0}^m\Phi_j(t,\epsilon)u_{j+1}(t,\epsilon)$.

Note that $\varepsilon=\epsilon^{\tilde{c}}$. Let
$$\widehat{\mathcal{R}}(\varepsilon)=\{\tilde{\varepsilon}\in(0,\varepsilon): \tilde{\varepsilon}^{\frac{1}{\tilde{c}}}\in\widetilde{\mathcal{R}}(\varepsilon^{\frac{1}{\tilde{c}}})\},$$
$\varepsilon_1=\epsilon_1^{\tilde{c}}$ and $\mathcal{E}=(0,\varepsilon_1)\backslash \widehat{\mathcal{R}}(\varepsilon_1)$.
It follows from Lemma \ref{lem1} that for each $\varepsilon\in \mathcal{E}$,
$$x=\Phi\left(t,\varepsilon^\frac{1}{\tilde{c}}\right) z_{\infty}+\Psi\left(t,\varepsilon^\frac{1}{\tilde{c}}\right)+x_0+\varepsilon^a u\left(t,\Phi\left(t,\varepsilon^\frac{1}{\tilde{c}}\right) z_{\infty}+\Psi\left(t,\varepsilon^\frac{1}{\tilde{c}}\right)+x_0\right)$$
transforms system \eqref{01.5} into \eqref{2.55}.
Since each iteration is a quasi-periodic transformation with frequency $\omega$, the zero solution of system \eqref{2.55} corresponds to a quasi-periodic solution $x^*(t,\varepsilon)$ of system \eqref{01.5}.
By \eqref{2.27} and \eqref{2.28}, we have $\lim\limits_{\varepsilon\rightarrow0}\Psi(t,\varepsilon^\frac{1}{\tilde{c}})=0$
and $\lim\limits_{
\varepsilon\rightarrow0}x^*(t,\varepsilon)=x_0$.

Finally, we estimate the relative measure of the set $\mathcal{E}\cap(0,\varepsilon)$ in $(0,\varepsilon)$ for $0<\varepsilon\leq \varepsilon_1$.
It is easy to see that for each $0<\varepsilon\leq\varepsilon_1$,
$$\mathcal{E}\cap(0,\varepsilon)=(0,\varepsilon)\backslash \widehat{\mathcal{R}}(\varepsilon),$$
and
$$\widehat{\mathcal{R}}(\varepsilon)\cap\left(\delta_1^{\tilde{c}}\varepsilon^{\delta_2},\varepsilon\right)
=\left\{\tilde{\varepsilon}\in\left(\delta_1^{\tilde{c}}\varepsilon^{\delta_2},\varepsilon\right):\tilde{\varepsilon}^{\frac{1}{\tilde{c}}}
\in\left(\widetilde{\mathcal{R}}(\varepsilon^{\frac{1}{\tilde{c}}})
\cap\left(\delta_1\varepsilon^{\frac{\delta_2}{\tilde{c}}},\varepsilon^{\frac{1}{\tilde{c}}}\right)\right)\right\}.$$
For $\tilde{\varepsilon}_1, \tilde{\varepsilon}_2\in(0,\varepsilon)$, we have
$$|\tilde{\varepsilon}_1^{\frac{1}{\tilde{c}}}-\tilde{\varepsilon}_2^{\frac{1}{\tilde{c}}}|
\geq\frac{1}{\tilde{c}}\varepsilon^{\frac{1}{\tilde{c}}-1}|\tilde{\varepsilon}_1-\tilde{\varepsilon}_2|.$$
It follows from \eqref{2.54}that
\begin{align*}
\meas\left(\widehat{\mathcal{R}}(\varepsilon)\cap\left(\delta_1^{\tilde{c}}\varepsilon^{\delta_2},\varepsilon\right)\right)\leq
\frac{(n+n^2)\mu^*\tilde{c}\varepsilon^{1-\frac{1}{\tilde{c}}}}{\kappa_2 M_{10}}\exp\left({-\frac{a_1}{(\mu^*\varepsilon^{a})^{a_2}}}\right),
\end{align*}
Hence,
\begin{align*}
\lim\limits_{\varepsilon\rightarrow0}\frac{\meas\left(\mathcal{E}\cap(0,\varepsilon)\right)}{\varepsilon}\geq 1-\lim\limits_{\varepsilon\rightarrow0}\delta_1^{\tilde{c}}\varepsilon^{\delta_2-1}-
\lim\limits_{\varepsilon\rightarrow0}\frac{\meas\left(\widehat{\mathcal{R}}(\varepsilon)\cap\left(\delta_1^{\tilde{c}}\varepsilon^{\delta_2},\varepsilon\right)\right)}
{\varepsilon}=1,
\end{align*}
which completes the proof.
\end{proof}

\section{Proof of Theorem \ref{theorem}}\label{sec4}

In this section, we  give a  proof of Theorem \ref{theorem}.
\begin{proof}
Using the rescaled parameter $\varepsilon=\epsilon^{2}$, system \eqref{1.1} can be rewritten as
\begin{align}\label{3.71}
\dot{x}=\epsilon^2 f_0(\omega t,x)+\epsilon^3 g(\omega t,x,\epsilon),
\end{align}
where $g(\omega t,x,\epsilon)=\frac{1}{\epsilon}\left(f(\omega t,x,\epsilon^2)-f_0(\omega t,x)\right)$. Clearly, both $f_0$ and $g$ are real analytic on
$\mathbb{T}^d\times U$.
Define $g_1(\omega t,x,\epsilon^2)=f(\omega t,x,\epsilon^2)-f_0(\omega t,x)$.
By the Lipschitz continuity of $f(\omega t,x,\varepsilon)$ on $[0,\varepsilon_0)$, we have
$$\|g_1(\omega t,x,\epsilon^2)\|\leq L\epsilon^2,$$
for all $\epsilon\in(0,\epsilon_0)$  and $(\omega t,x)\in \mathbb{T}^d\times U$,
where $\epsilon_0=\sqrt{\varepsilon_0}$ and $L$ is the Lipschitz constant of $f$.
For any $0<\tilde{\epsilon}_1<\tilde{\epsilon}_2<\epsilon_0$,
\begin{align*}
&\|g(\omega t,x,\tilde{\epsilon}_1)-g(\omega t,x,\tilde{\epsilon}_2)\|\\
&=\left\|\frac{g_1(\omega t,x,\tilde{\epsilon}_1^2)}{\tilde{\epsilon}_1}-\frac{g_1(\omega t,x,\tilde{\epsilon}_2^2)}{\tilde{\epsilon}_2}\right\|\\
&\leq\left\|\frac{\tilde{\epsilon}_2 g_1(\omega t,x,\tilde{\epsilon}_1^2)-\tilde{\epsilon}_1 g_1(\omega t,x,\tilde{\epsilon}_1^2)}{\tilde{\epsilon}_1\tilde{\epsilon}_2}\right\|+
\left\|\frac{\tilde{\epsilon}_1 g_1(\omega t,x,\tilde{\epsilon}_1^2)-\tilde{\epsilon}_1 g_1(\omega t,x,\tilde{\epsilon}_2^2)}{\tilde{\epsilon}_1\tilde{\epsilon}_2}\right\|\\
&\leq3L|\tilde{\epsilon}_1-\tilde{\epsilon}_2|.
\end{align*}
Similarly, one can obtain the Lipschitz continuity of $\partial_x g(\omega t,x,\epsilon)$.
Then, according to Theorem \ref{th1.2}, there exist a constant $\epsilon_1>0$ and
 a subset $\widehat{\mathcal{R}}(\epsilon_1)\subset (0,\epsilon_1)$ such that
for each $\epsilon\in(0,\epsilon_1)\backslash\widehat{\mathcal{R}}(\epsilon_1)$, system \eqref{3.71} has a response solution.
Moreover, we have $\lim\limits_{\epsilon\rightarrow0}\frac{\meas(\widehat{\mathcal{R}}(\epsilon))}{\epsilon}=0$.

Let $\varepsilon_1=\epsilon_1^2$ and define
$$\overline{\mathcal{R}}(\varepsilon)=\{\tilde{\varepsilon}\in(0,\varepsilon): \sqrt {\tilde{\varepsilon}}\in\widehat{\mathcal{R}}(\sqrt{\varepsilon})\},$$
for $0<\varepsilon\leq\varepsilon_1$.
Let $\mathcal{E}=(0,\varepsilon_1)\backslash \overline{\mathcal{R}}(\varepsilon_1)$.
Then for each $\varepsilon\in\mathcal{E}$, system \eqref{1.1} has a response solution.
For any $\tilde{\varepsilon}_1, \tilde{\varepsilon}_2\in(0,\varepsilon)$, we have
$$\left|\sqrt{\tilde{\varepsilon}_1}-\sqrt{\tilde{\varepsilon}_2}\right|
\geq\frac{1}{2\sqrt{\varepsilon}}\left|\tilde{\varepsilon}_1-\tilde{\varepsilon}_2\right|.$$
It follows that
$$\meas(\overline{\mathcal{R}}(\varepsilon))\leq 2\sqrt{\varepsilon}\meas(\widehat{\mathcal{R}}(\sqrt\varepsilon)).$$
Note that
$\mathcal{E}\cap(0,\varepsilon)=(0,\varepsilon)\backslash\overline{\mathcal{R}}(\varepsilon).$
Then,
\begin{align*}
\lim\limits_{\varepsilon\rightarrow0}\frac{\meas\left(\mathcal{E}\cap(0,\varepsilon)\right)}{\varepsilon}&\geq 1-
\lim\limits_{\varepsilon\rightarrow0}\frac{\meas\left(\overline{\mathcal{R}}(\varepsilon)\right)}
{\varepsilon}\\
&\geq 1-
\lim\limits_{\varepsilon\rightarrow0}\frac{2\meas(\widehat{\mathcal{R}}(\sqrt\varepsilon))}
{\sqrt{\varepsilon}}\\
&=1.
\end{align*}

\end{proof}

\section{Higher-order averaging method}\label{sec5}
In this section, we present a higher-order averaging method for establishing the existence of response solutions in quasi-periodic systems. For higher-order averaging methods concerning periodic solutions and invariant tori of periodic systems, we refer to \cite{GLWZ,LNT,NP,PNC}.

For $r>0$ and an open bounded subset $U\subset\mathbb{R}^n$, let
$\mathbf{B}_r(U)=\{x\in\mathbb{C}^n:\ \inf\limits_{z\in U}\|x-z\|\leq r\}$. Let $N$ be a fixed positive integer.
Consider the higher-order perturbed system
\begin{align}\label{4}
\dot{x}=\sum\limits_{k=1}^N\varepsilon^k f_k(\omega t,x)+\varepsilon^{N+1}g(\omega t,x,\varepsilon),
\end{align}
where $f_k:\mathbb{T}^d_\rho\times \mathbf{B}_r(U)\rightarrow \mathbb{C}^n$ is real analytic for $k=1,\dots, N$, $g:\mathbb{T}^d_\rho\times \mathbf{B}_r(U)\times(0,\varepsilon_0)\rightarrow \mathbb{C}^n$ is continuous and real analytic on $\mathbb{T}^d_\rho\times \mathbf{B}_r(U)$ for each $\varepsilon\in(0,\varepsilon_0)$. Furthermore, $g(\theta,x,\varepsilon)$ and $\partial_x g(\theta,x,\varepsilon)$ are Lipschitz continuous with respect to $\varepsilon$ on $(0,\varepsilon_0)$ uniformly for $(\theta,x)\in\mathbb{T}^d_\rho\times \mathbf{B}_r(U)$.

Assume that there exists a constant $M>0$ such that $\sup\limits_{x\in\mathbf{B}_{r}(U)}\|f_k(\cdot,x)\|_{\rho}\leq M$ for $k=1, \dots, N$ and
$\sup\limits_{\substack{x\in\mathbf{B}_{r}(U)\\ \varepsilon\in(0,\varepsilon_0)}}\|g(\cdot,x,\varepsilon)\|_{\rho}\leq M$.
Let $\tilde{\rho}_k=\frac{(2N-k)\rho}{2N}$ and $\tilde{r}_k=\frac{(2N-k) r}{2N}$ for $k=1, \dots, N$.

The following lemma indicates that via a near-identity transformation, system \eqref{4} can be transformed into the form of an autonomous system plus a higher-order perturbation.
\begin{lemma}\label{lem 4.1}
Assume that $\omega$ satisfies the Diophantine condition \eqref{DC}. Then there exist a constant $0<\varepsilon_1<\varepsilon_0$ and a change of variables $x= \Xi^N(\omega t,y,\varepsilon)$ defined for $\varepsilon\in(0,\varepsilon_1)$, that transforms the system \eqref{4} into
\begin{align}\label{4.78}
\dot{y}=\sum_{k=1}^N\varepsilon^k Y_k(y)+\varepsilon^{N+1}g_N(\omega t,y,\varepsilon).
\end{align}
Here, the transformation is given by $\Xi^N=\Xi_1\circ\Xi_2\circ\cdots\circ\Xi_N$, where  $\Xi_k(\omega t,y,\varepsilon)=y+\varepsilon^ku_k(\omega t,y)$ with
$u_k$ real analytic on $\mathbb{T}^d_{\tilde{\rho}_k}\times\mathbf{B}_{\tilde{r}_k}(U)$ for $k=1, \dots, N$;
each $Y_k:\mathbf{B}_{r/2}(U)\to \mathbb{C}^n$ is real analytic for $k=1, \dots, N$; $g_N$ is continuous and real analytic on $\mathbb{T}^d_{\rho/2}\times\mathbf{B}_{r/2}(U)$ for every fixed $\varepsilon\in(0,\varepsilon_1)$; and $g_N(\theta,y,\varepsilon)$ and $\partial_y g_N(\theta,y,\varepsilon)$ are Lipschitz continuous with respect to $\varepsilon$ on $(0,\varepsilon_1)$ uniformly for $(\theta,y)\in\mathbb{T}^d_{\rho/2}\times\mathbf{B}_{r/2}(U)$. Furthermore, the functions $Y_k$ and $u_k$ (for $k=1, \dots, N$) can be computed in the sequence $Y_1, u_1, Y_2, u_2, \dots, Y_N, u_N$.
\end{lemma}

In \cite{S}, Sim\'{o} established a result analogous to Lemma \ref{lem 4.1} that makes the time-dependent remainder exponentially small. However, our objective is to present recursive formulas for the averaged functions at each order. To ensure the completeness of the paper, we provide a direct proof of Lemma \ref{lem 4.1} at the end of this section.

\begin{remark}
 As seen in the proof, \eqref{4.92}, \eqref{4.89} and \eqref{4.93} provide recursive expressions for the functions $Y_m$ and $u_m$ ($1\leq m \leq N$).
For clarity, we outline the calculation of the first three terms $Y_m$ as follows:
 \begin{align*}
 &Y_1=\overline{f_1},\\
 &u_1(\omega t,x)=\int_0^t\widetilde{f_1}(\omega s,x)\mathrm{d}s-\lim\limits_{T\rightarrow+\infty}\frac{1}{T}\int_0^T\int_0^t\widetilde{f_1}(\omega s,x)\mathrm{d}s\mathrm{d}t,\\
 &f_{1,2}=-(\partial_x u_1) Y_1+(\partial_x f_1)u_1+f_2, \\
 &Y_2=\overline{f_{1,2}},\\
 &f_{1,3}=(\partial_x u_1)^2Y_1+\frac{1}{2}(\partial^2_xf_1) u_1^2-(\partial_x u_1)(\partial_x f_1) u_1
 +(\partial_x f_2) u_1-(\partial_x u_1)f_2 +f_3,\\
 &u_2(\omega t,x)=\int_0^t\widetilde{f_{1,2}}(\omega s,x)\mathrm{d}s-\lim\limits_{T\rightarrow+\infty}\frac{1}{T}\int_0^T\int_0^t\widetilde{f_{1,2}}(\omega s,x)\mathrm{d}s\mathrm{d}t,\\
 &f_{2,3}=-(\partial_x u_2) Y_1+(\partial_x Y_1) u_2+f_{1,3},\\
 &Y_3=\overline{f_{2,3}}.
 \end{align*}
\end{remark}

Based on Lemma \ref{lem 4.1} and Theorem \ref{th1.2}, we can immediately obtain the following higher-order averaging result concerning the existence of response solutions.
\begin{theorem}
Assume that $\omega$ satisfies the Diophantine condition \eqref{DC}. For some $l \in \{1, \dots, N\}$, suppose that $Y_j = 0$ for all $j = 1, \dots, l-1$ (where by convention we set $Y_0 = 0$).
Moreover, suppose that
there exists a point $x_0\in U$ such that $Y_l(x_0)=0$ and the eigenvalues of $\partial_xY_l(x_0)$ are nonzero and pairwise distinct. Then
there exist  a constant $0<\varepsilon_1<\varepsilon_0$ and a Cantorian set $\mathcal{E}\subset(0,\varepsilon_1)$ such that for each $\varepsilon\in\mathcal{E}$, the system
\eqref{4} has a quasi-periodic response solution $x^*(t,\varepsilon)$ satisfying $\lim\limits_{\substack{\varepsilon \in \mathcal{E} \\ \varepsilon \rightarrow 0}}x^*(t,\varepsilon)=x_0$ uniformly for $t\in\mathbb{R}$. The relative Lebesgue measure of $\mathcal{E}\cap(0,\varepsilon)$ with respect to $(0,\varepsilon)$ tends to 1 as $\varepsilon\rightarrow0$.
\end{theorem}
\begin{remark}
If the real parts of all eigenvalues of \( \partial_xY_l(x_0) \) are nonzero, then according to Lemma \ref{lem 4.1} and the classical averaging method (see, e.g., Hale \cite{Hale}), for sufficiently small \( \varepsilon_1>0 \), system \eqref{4} admits a quasi-periodic response solution $x^*(t,\varepsilon)$ for any \( \varepsilon \in (0, \varepsilon_1) \). Furthermore, if all eigenvalues of matrix \( \partial_xY_l(x_0) \) have negative real parts, then $x^*(t,\varepsilon)$ is asymptotically stable; if there exists an eigenvalue of \( \partial_xY_l(x_0) \) with a positive real part, then $x^*(t,\varepsilon)$ is unstable.
\end{remark}

In the remainder of this section, for a real analytic function
$f=(f^{(1)}, \dots, f^{(n)})^\top:\mathbb{T}^d_\rho\times \mathbf{B}_r(U)\rightarrow\mathbb{C}^n,$
let $\partial^k f(\theta,x)$, for $k \in \mathbb{N}$,
denote the $k$-th order partial derivative of $f$ with respect to $x$ at $(\theta,x)$.
$\partial^k f(\theta,x)$ is a $k$-linear map on $\mathbb{C}^n$ for each $(\theta,x)\in \mathbb{T}^d_\rho\times \mathbf{B}_r(U)$. For
$\xi_i=(\xi_i^{(1)},\dots,\xi_i^{(n)})^\top\in\mathbb{C}^n$, $i=1,\dots,k$,
$$\partial^k f(\theta,x)\xi_1\cdots\xi_k=\left(\partial^k f^{(1)}(\theta,x)\xi_1\cdots\xi_k,\dots,\partial^k f^{(n)}(\theta,x)\xi_1\cdots\xi_k\right)^\top,$$
where
$$\partial^k f^{(j)}(\theta,x)\xi_1\cdots\xi_k=\sum\limits_{i_1=1}^n\cdots\sum\limits_{i_k=1}^n\frac{\partial^kf^{(j)}(\theta,x)}{\partial x^{(i_1)}\cdots \partial x^{(i_k)}}\xi_1^{(i_1)}\cdots \xi_k^{(i_k)},$$
for $j=1,\dots, n$.
When $\xi_i=\xi$ for $i=1,\dots,k$, we adopt the abbreviated notation
$$\partial^k f(\theta,x)\xi^k:=\partial^k f(\theta,x)\xi_1\cdots\xi_k.$$
Note that $\xi^{(k)}$ denotes the $k$-th component of $\xi$, while $\xi^k$ is shorthand for $\xi\cdots\xi$ ($k$ times). We adopt the convention that $\partial^0 f=f$.

Using these notations, according to Taylor's formula, for any $m\in\mathbb{N}$, $x\in\mathbf{B}_r(U)$, $u \in \mathbb{C}^n$ and $\varepsilon \in [0, \varepsilon_0)$ such that $x+\varepsilon u\in\mathbf{B}_r(U)$, the real analytic function $f(\theta, x + \varepsilon u)$ can be expressed as:
\begin{align*}
f(\theta, x + \varepsilon u) = \sum_{k=0}^{m} \frac{\varepsilon^{k}}{k!} \partial^k f(\theta, x) u^{k} + \varepsilon^{m+1} R_{m,f}(\theta,x,u,\varepsilon),
\end{align*}
where
$$R_{m,f}(\theta,x,u,\varepsilon)=\frac{1}{m!} \int_0^1 (1-s)^m \partial^{m+1} f(\theta, x + s\varepsilon u) u^{m+1}  ds.$$

Now we give the proof of Lemma \ref{lem 4.1}.
\begin{proof}[Proof of Lemma \ref{lem 4.1}]
We employ a recursive approach for the proof.

First, consider the equation
\begin{align}\label{4.79}
\partial_t u_1(\omega t,x)=\widetilde{f_1}(\omega t,x),\ \ (\omega t,x)\in \mathbb{T}^d_{\tilde{\rho}_1}\times\mathbf{B}_{r}(U).
\end{align}
Assume that the Fourier expansions of $u_1$ and $f_1$ take the form:
$$u_1(\omega t,x)=\sum\limits_{k\in \mathbb{Z}^d}\hat{u}_{1,k}(x)e^{\sqrt{-1}\langle k,\omega\rangle t},\ \ \ f_1(\omega t,x)=\sum\limits_{k\in \mathbb{Z}^d}\hat{f}_{1,k}(x)e^{\sqrt{-1}\langle k,\omega\rangle t}.$$
Comparing the coefficients of each order, one has
$$\hat{u}_{1,k}(x)=\frac{\hat{f}_{1,k}(x)}{\sqrt{-1}\langle k,\omega\rangle},\ \ \ \ \ k\in \mathbb{Z}^d\backslash\{0\}.$$
Take $\hat{u}_{1,0}=0$. Then by the Diophantine condition \eqref{DC} and Cauchy's formula, we have
\begin{align}\label{4.80}
\|u_1(\cdot,x)\|_{\tilde{\rho}_1}
&\leq \frac{1}{\gamma}\|f_1(\cdot,x)\|_{\rho}\sum\limits_{k\in\mathbb{Z}^d\backslash\{0\}}e^{-\frac{\rho}{2N} |k|}|k|^\tau\leq M \varpi_1
\end{align}
for all $x\in\mathbf{B}_{r}(U)$, where $\varpi_1=\frac{1}{\gamma}\sum\limits_{k\in\mathbb{Z}^d\backslash\{0\}}e^{-\frac{\rho}{2N} |k|}|k|^\tau$ is
well defined by Lemma \ref{ALe2}.
Moreover,
\begin{align}\label{4.81}
\|\partial u_1(\cdot,x)\|_{\tilde{\rho}_1}\leq \sup\limits_{x\in\mathbf{B}_r(U)}\frac{2N}{r}\|u_1(\cdot,x)\|_{\tilde{\rho}_1}\leq\frac{2NM \varpi_1}{r}
\end{align}
for $x\in\mathbf{B}_{\tilde{r}_1}(U)$.
Setting $\varepsilon_1\leq\frac{r}{4NM\varpi_1}$, we see that $I_n+\varepsilon \partial u_1(\omega t,x)$ is invertible for $0<\varepsilon<\varepsilon_1$, and
\begin{align}\label{04.82}
\|\left(I_n+\varepsilon \partial u_1\right)^{-1}\|_{\tilde{\rho}_1}\leq\frac{1}{1-\varepsilon\|\partial u_1(\cdot,x)\|_{\tilde{\rho}_1}}\leq 2
\end{align}
for $x\in\mathbf{B}_{\tilde{r}_1}(U)$ and $\varepsilon\in(0,\varepsilon_1)$.
Then, by making the change of variables
$$x=y_1+\varepsilon u_1(\omega t,y_1),\ \ \ (\omega t,y_1)\in \mathbb{T}^d_{\tilde{\rho}_1}\times\mathbf{B}_{\tilde{r}_1}(U),$$
we obtain
\begin{align*}
\left(I_n+\varepsilon \partial u_1\right)\dot{y}_1+\varepsilon\partial_t u_1
&=\varepsilon\overline{f_1}(y_1)+\varepsilon\widetilde{f_1}(\omega t,y_1)+\varepsilon\left(f_1(\omega t,y_1+\varepsilon u_1)-f_1(\omega t,y_1)\right)\\
&+\sum\limits_{k=2}^N\varepsilon^k f_k(\omega t, y_1+\varepsilon u_1)
+\varepsilon^{N+1}g(\omega t,y_1+\varepsilon u_1,\varepsilon).
\end{align*}
It follows from \eqref{4.79} that
\begin{align*}
\dot{y}_1
&=\varepsilon\left(I_n+\varepsilon \partial u_1\right)^{-1}\overline{f_1}(y_1)\\
&+\varepsilon\left(I_n+\varepsilon \partial u_1\right)^{-1}\left({f_1}(\omega t,y_1+\varepsilon u_1)-{f_1}(\omega t,y_1)\right)\\
&+\sum\limits_{k=2}^N\varepsilon^k \left(I_n+\varepsilon \partial u_1\right)^{-1}f_k(\omega t, y_1+\varepsilon u_1)\\
&+\varepsilon^{N+1}\left(I_n+\varepsilon \partial u_1\right)^{-1}g(\omega t,y_1+\varepsilon u_1,\varepsilon).
\end{align*}
Note that for $m\geq 1$,
\begin{align}
\left(I_n+\varepsilon \partial u_1\right)^{-1}&=I_n+\sum\limits_{k=1}^m(-1)^k\varepsilon^k (\partial u_1)^k+\varepsilon^{m+1}
S_m(\partial u_1,\varepsilon),
\end{align}
where
\begin{align}\label{04.84}
S_m(\partial u_1,\varepsilon)=(-1)^{m+1}\left(I_n+ \varepsilon\partial u_1\right)^{-1}(\partial u_1)^{m+1}.
\end{align}
Then by Taylor's formula, we obtain
\begin{align}\label{04.85}
\dot{y}_1
=\varepsilon Y_1(y_1)+\sum\limits_{k=2}^N\varepsilon^k f_{1,k}(\omega t, y_1)+\varepsilon^{N+1}g_1(\omega t, y_1,\varepsilon),
\end{align}
where $Y_1=\overline{f_1}$,
\begin{align}\label{4.84}
f_{1,k}(\omega t,y_1)&=(-1)^{k-1}(\partial u_1)^{k-1}\overline{f_1}(y_1)\nonumber\\
&+\sum\limits_{j=0}^{k-2}\frac{(-1)^{j}}{(k-j-1)!}(\partial u_1)^{j}\partial ^{k-j-1}{f_1}(\omega t,y_1)u_1^{k-j-1}\nonumber\\
&+\sum\limits_{i=2}^k\sum\limits_{j=0}^{k-i}\frac{(-1)^{j}}{(k-j-i)!}(\partial u_1)^{j}\partial ^{k-j-i}f_i(\omega t,y_1)u_1^{k-j-i},
\end{align}
for $k=2, \dots, N$ and
\begin{align}\label{4.85}
&g_1(\omega t, y_1,\varepsilon)\nonumber\\
&=S_{N-1}(\partial u_1,\varepsilon)\overline{f_1}(y_1)
+\sum\limits_{j_1=0}^{N}\sum\limits_{j_2=N-j_1}^{N}\varepsilon^{j_1+j_2-N}\frac{(-1)^{j_1}}{j_2!}(\partial u_1)^{j_1}\partial ^{j_2}{f_1}(\omega t,y_1)u_1^{j_2}\nonumber\\
&+\varepsilon\left(I_n+ \varepsilon\partial u_1\right)^{-1}R_{N,f_1}(\omega t,y_1,u_1,\varepsilon)
+\varepsilon S_N(\partial u_1,\varepsilon)\left({f_1}(\omega t,y_1+\varepsilon u_1)-{f_1}(\omega t,y_1)\right)\nonumber\\
&-\varepsilon^{N+2}S_N(\partial u_1,\varepsilon)R_{N,f_1}(\omega t,y_1,u_1,\varepsilon)\nonumber\\
&+\sum_{k=2}^N\sum\limits_{j_1=0}^{N}\sum\limits_{j_2=l(j_1,k)}^{N}\varepsilon^{j_1+j_2+k-N-1}\frac{(-1)^{j_1}}{j_2!}(\partial u_1)^{j_1}\partial ^{j_2}{f_k}(\omega t,y_1)u_1^{j_2}\nonumber\\
&+\sum_{k=2}^N\varepsilon^k\left(\left(I_n+ \varepsilon\partial u_1\right)^{-1}R_{N,f_k}(\omega t,y_1,u_1,\varepsilon)+S_N(\partial u_1,\varepsilon){f_k}(\omega t,y_1+\varepsilon u_1)\right)\nonumber\\
&-\sum_{k=2}^N\varepsilon^{N+k+1}S_N(\partial u_1,\varepsilon)R_{N,f_k}(\omega t,y_1,u_1,\varepsilon)+\left(I_n+\varepsilon \partial u_1\right)^{-1}g(\omega t,y_1+\varepsilon u_1,\varepsilon),
\end{align}
with $l(j_1,k)=\max\{N-k-j_1+1,0\}$.

Since $\sup\limits_{x\in\mathbf{B}_{r}(U)}\|f_k(\cdot,x)\|_{\rho}\leq M$ for each $1\leq k\leq N$, it can be proved using \eqref{4.80}, \eqref{4.81}, \eqref{4.84} and Cauchy estimate that there exists a constant $M_1>0$ such that
$\sup\limits_{y_1\in\mathbf{B}_{\tilde{r}_1}(U)}\|f_{1,k}(\cdot,y_1)\|_{\tilde{\rho}_1}\leq M_1$ for $k=2, \dots, N$.
From \eqref{4.80}, \eqref{4.81}, \eqref{04.82}, \eqref{4.85} and Cauchy estimate, we see that the function $g_1$ is also bounded. Without loss of generality, we still assume that
$$\sup\limits_{\substack{y_1\in\mathbf{B}_{\tilde{r}_1}(U)\\ \varepsilon\in(0,\varepsilon_1)}}\|g_1(\cdot,y_1,\varepsilon)\|_{\tilde{\rho}_1}\leq M_1.$$
According to the analyticity of $f_k$ and $g$, and the Lipschitz continuity of $g$ and $\partial g$ with respect to $\varepsilon$, it can be inferred that $g_1$ and $\partial  g_1$ are also Lipschitz continuous with respect to $\varepsilon$ uniformly on $\mathbb{T}^d_{\tilde{\rho}_1}\times\mathbf{B}_{\tilde{r}_1}(U)$.

Now assume that after $m$ steps of transformation ($1 \leq m \leq N-1$), system \eqref{4} is transformed into the following form:
\begin{align}\label{4.87}
\dot{y}_m
=\sum\limits_{k=1}^m\varepsilon^k Y_k(y_m)+\sum\limits_{k=m+1}^N\varepsilon^k f_{m,k}(\omega t, y_m)+\varepsilon^{N+1}g_m(\omega t, y_m,\varepsilon),
\end{align}
where $Y_k$ is real analytic on $\mathbf{B}_{\tilde{r}_k}(U)$ for $k=1, \dots, m$, $f_{m,k}$ is real analytic on $\mathbb{T}^d_{\tilde{\rho}_m}\times\mathbf{B}_{\tilde{r}_m}(U)$ for $k=m+1, \dots, N$ and $g_m$ is continuous and real analytic on $\mathbb{T}^d_{\tilde{\rho}_m}\times\mathbf{B}_{\tilde{r}_m}(U)$ for each $\varepsilon\in(0,\varepsilon_1)$. $g_m$ and $\partial g_m$ are Lipschitz continuous with respect to $\varepsilon$ uniformly on $\mathbb{T}^d_{\tilde{\rho}_m}\times\mathbf{B}_{\tilde{r}_m}(U)$.
Moreover, there exists a constant $M_m>0$ such that
$\sup\limits_{y_m\in\mathbf{B}_{\tilde{r}_m}(U)}\|f_{m,k}(\cdot,y_m)\|_{\tilde{\rho}_m}\leq M_m$ for $k=m+1, \dots, N$ and
$$\sup\limits_{\substack{y_m\in\mathbf{B}_{\tilde{r}_m}(U)\\ \varepsilon\in(0,\varepsilon_1)}}\|g_m(\cdot,y_m,\varepsilon)\|_{\tilde{\rho}_m}\leq M_m.$$

We present the transformation at step $m+1$.
Let $u_{m+1}(\omega t,x)$ be the solution of
 \begin{align}\label{4.89}
\partial_t u_{m+1}(\omega t,x)=\widetilde{f_{m,m+1}}(\omega t,x),\ \ (\omega t,x)\in \mathbb{T}^d_{\tilde{\rho}_{m+1}}\times\mathbf{B}_{\tilde{r}_{m}}(U),
\end{align}
with $\overline{u_{m+1}}=0$. Then, similar to \eqref{4.80} and \eqref{4.81}, one has
\begin{align}
\|u_{m+1}(\cdot,x)\|_{\tilde{\rho}_{m+1}}\leq M_m \varpi_1
\end{align}
for $x\in\mathbf{B}_{\tilde{r}_m}(U)$, and
\begin{align}
\|\partial u_{m+1}(\cdot,x)\|_{\tilde{\rho}_{m+1}}\leq\frac{2NM_m \varpi_1}{r}
\end{align}
for $x\in\mathbf{B}_{\tilde{r}_{m+1}}(U)$.
Then, for sufficiently small $\varepsilon_1>0$, by performing the variable substitution
$$y_m=y_{m+1}+\varepsilon^{m+1} u_{m+1}(\omega t,y_{m+1}),\ \ \ (\omega t,y_{m+1},\varepsilon)\in \mathbb{T}^d_{\tilde{\rho}_{m+1}}\times\mathbf{B}_{\tilde{r}_{m+1}}(U)\times(0,\varepsilon_1),$$
on system \eqref{4.87}, we obtain
\begin{align*}
&\left(I_n+\varepsilon^{m+1} \partial  u_{m+1}\right)\dot{y}_{m+1}+\varepsilon^{m+1}\partial_t u_{m+1}\\
&=\sum\limits_{k=1}^m\varepsilon^k Y_k\left(y_{m+1}+\varepsilon^{m+1} u_{m+1}\right)+\varepsilon^{m+1}\overline{f_{m,m+1}}(y_{m+1})+\varepsilon^{m+1}\widetilde{f_{m,{m+1}}}(\omega t,y_{m+1})\\
&+\varepsilon^{m+1}\left(f_{m,m+1}(\omega t,y_{m+1}+\varepsilon^{m+1} u_{m+1})-f_{m,m+1}(\omega t,y_{m+1})\right)\\
&+\sum\limits_{k=m+2}^N\varepsilon^k f_{m,k}(\omega t, y_{m+1}+\varepsilon^{m+1} u_{m+1})
+\varepsilon^{N+1}g_m(\omega t, y_{m+1}+\varepsilon^{m+1} u_{m+1},\varepsilon).
\end{align*}
It follows that
\begin{align*}
&\dot{y}_{m+1}\\
&=\sum\limits_{k=1}^m\varepsilon^k \left(I_n+\varepsilon^{m+1} \partial u_{m+1}\right)^{-1}Y_k\left(y_{m+1}+\varepsilon^{m+1} u_{m+1}\right)\\
&+\varepsilon^{m+1}\left(I_n+\varepsilon^{m+1} \partial u_{m+1}\right)^{-1}\overline{f_{m,m+1}}(y_{m+1})
+\varepsilon^{m+1}\left(I_n+\varepsilon^{m+1} \partial u_{m+1}\right)^{-1}\\
&~~~\cdot\left(f_{m,m+1}(\omega t,y_{m+1}+\varepsilon^{m+1} u_{m+1})-f_{m,m+1}(\omega t,y_{m+1})\right)\\
&+\sum\limits_{k=m+2}^N\varepsilon^k \left(I_n+\varepsilon^{m+1} \partial u_{m+1}\right)^{-1}f_{m,k}(\omega t, y_{m+1}+\varepsilon^{m+1} u_{m+1})\\
&+\varepsilon^{N+1}\left(I_n+\varepsilon^{m+1} \partial u_{m+1}\right)^{-1}g_m(\omega t, y_{m+1}+\varepsilon^{m+1} u_{m+1},\varepsilon).
\end{align*}
For $s\in\mathbb{R}$, let $\lfloor s\rfloor$ denote the greatest integer less than or equal to
$s$.
For $m, j, k\in\mathbb{N}$ and $i=0,1$, define the sets
\[\mathcal{C}_{m,j}^i(k)=\left\{(j_1,j_2)\in\mathbb{N}^2: (j_1+j_2)(m+1)+j=k, j_1\leq c_{m,j},  i\leq j_2\leq c_{m,j}\right\},\]
and
\[\mathcal{C}_{m,j}^i=\left\{(j_1,j_2)\in\mathbb{N}^2: (j_1+j_2)(m+1)+j\geq N+1,  j_1\leq c_{m,j}, i\leq j_2\leq c_{m,j}\right\},\]
where $c_{m,j}=\left\lfloor\frac{N-j}{m+1}\right\rfloor$.
Define the function $\delta:\mathbb{R}\rightarrow\mathbb{R}$ such that
\[\delta(x)=
\begin{cases}
1& \text{if } x\geq1, \\
0 & \text{otherwise}.
\end{cases}
\]
For $m, k\in\mathbb{N}$, let
\[
l^*(m, k) =
\begin{cases}
\frac{k}{m+1} - 1 & \text{if } \frac{k}{m+1} - 1 \in \mathbb{Z}^+, \\
0 & \text{otherwise}.
\end{cases}
\]
Following a derivation similar to that of \eqref{04.85}, by applying Taylor's formula, we obtain
\begin{align*}
\dot{y}_{m+1}
=\sum\limits_{k=1}^{m+1}\varepsilon^k Y_k(y_{m+1})+\sum\limits_{k=m+2}^N\varepsilon^k f_{m+1,k}(\omega t, y_{m+1})+\varepsilon^{N+1}g_{m+1}(\omega t, y_{m+1},\varepsilon),
\end{align*}
where
\begin{align}\label{4.92}
Y_{m+1}=\overline{f_{m,m+1}},
\end{align}
\begin{align}\label{4.93}
&f_{m+1,k}(\omega t,y_{m+1})\nonumber\\
&=\sum_{j=1}^m\sum\limits_{(j_1,j_2)\in\mathcal{C}_{m,j}^0(k)}\frac{(-1)^{j_1}}{j_2!}(\partial u_{m+1})^{j_1}\partial ^{j_2}{Y_j}(y_{m+1})u_{m+1}^{j_2}\nonumber\\
&+\delta(l^*(m, k))(-1)^{l^*(m, k)}(\partial u_{m+1})^{l^*(m, k)}\overline{f_{m,m+1}}\nonumber\\
&+\sum\limits_{(j_1,j_2)\in\mathcal{C}_{m,m+1}^1(k)}\frac{(-1)^{j_1}}{j_2!}(\partial u_{m+1})^{j_1}\partial ^{j_2}f_{m,m+1}(\omega t,y_{m+1})u_{m+1}^{j_2}\nonumber\\
&+\sum_{j=m+2}^k\sum\limits_{(j_1,j_2)\in\mathcal{C}_{m,j}^0(k)}\frac{(-1)^{j_1}}{j_2!}(\partial u_{m+1})^{j_1}\partial ^{j_2}
{f_{m,j}}(\omega t,y_{m+1})u_{m+1}^{j_2},
\end{align}
for $k=m+2, \dots, N$ and
\begin{align*}
&g_{m+1}(\omega t, y_{m+1},\varepsilon)\nonumber\\
&=\sum_{j=1}^m\sum\limits_{(j_1,j_2)\in\mathcal{C}_{m,j}^0}\varepsilon^{(m+1)(j_1+j_2)+j-N-1}\frac{(-1)^{j_1}}{j_2!}(\partial u_{m+1})^{j_1}
\partial ^{j_2}{Y_j}(y_{m+1})u_{m+1}^{j_2}\nonumber\\
&+\sum_{j=1}^m\varepsilon^{(m+1)(c_{m,j}+1)+j-N-1}\left(I_n+ \varepsilon^{m+1}\partial u_{m+1}\right)^{-1}R_{c_{m,j},Y_j}(\omega t,y_{m+1},u_{m+1},\varepsilon^{m+1})\nonumber\\
&+\sum_{j=1}^m\varepsilon^{(m+1)(c_{m,j}+1)+j-N-1}S_{c_{m,j}}(\partial u_{m+1},\varepsilon^{m+1}){Y_j}(y_{m+1}+\varepsilon^{m+1} u_{m+1})\nonumber\\
&-\sum_{j=1}^m\varepsilon^{2(m+1)(c_{m,j}+1)+j-N-1}S_{c_{m,j}}(\partial u_{m+1},\varepsilon^{m+1})R_{c_{m,j},Y_j}(\omega t,y_{m+1},u_{m+1},\varepsilon^{m+1})\nonumber\\
&+\varepsilon^{(m+1)(c_{m,m+1}+2)-N-1}S_{c_{m,m+1}}(\partial u_{m+1},\varepsilon^{m+1})\overline{f_{m,m+1}}(y_{m+1})\nonumber\\
&+\sum\limits_{(j_1,j_2)\in\mathcal{C}_{m,m+1}^1}\varepsilon^{(m+1)(j_1+j_2+1)-N-1}\frac{(-1)^{j_1}}{j_2!}(\partial u_{m+1})^{j_1}\partial ^{j_2}f_{m,m+1}(\omega t,y_{m+1})u_{m+1}^{j_2}\nonumber\\
&+\varepsilon^{(m+1)(c_{m,m+1}+2)-N-1}\left(I_n+\varepsilon^{m+1} \partial u_{m+1}\right)^{-1}R_{c_{m,m+1},f_{m,m+1}}(\omega t,y_{m+1},u_{m+1},\varepsilon^{m+1})\nonumber\\
&+\varepsilon^{(m+1)(c_{m,m+1}+2)-N-1}S_{c_{m,m+1}}(\partial u_{m+1},\varepsilon^{m+1})\\
&~~~\cdot \left(f_{m,m+1}(\omega t,y_{m+1}+\varepsilon^{m+1} u_{m+1})-f_{m,m+1}(\omega t,y_{m+1})\right)\\
&-\varepsilon^{2(m+1)(c_{m,m+1}+1)+m-N}S_{c_{m,m+1}}(\partial u_{m+1},\varepsilon^{m+1})\\
&~~~\cdot R_{c_{m,m+1},f_{m,m+1}}(\omega t,y_{m+1},u_{m+1},\varepsilon^{m+1})\\
&+\sum_{j=m+2}^N\sum\limits_{(j_1,j_2)\in\mathcal{C}_{m,j}^0}\varepsilon^{(m+1)(j_1+j_2)+j-N-1}\frac{(-1)^{j_1}}{j_2!}(\partial u_{m+1})^{j_1}
\partial ^{j_2}f_{m,j}(\omega t,y_{m+1})u_{m+1}^{j_2}\nonumber\\
&+\sum_{j=m+2}^N\varepsilon^{(m+1)(c_{m,j}+1)+j-N-1}\left(I_n+ \varepsilon^{m+1}\partial u_{m+1}\right)^{-1}R_{c_{m,j},f_{m,j}}(\omega t,y_{m+1},u_{m+1},\varepsilon^{m+1})\nonumber\\
&+\sum_{j=m+2}^N\varepsilon^{(m+1)(c_{m,j}+1)+j-N-1}S_{c_{m,j}}(\partial u_{m+1},\varepsilon^{m+1}){f_{m,j}}(y_{m+1}+\varepsilon^{m+1} u_{m+1})\nonumber\\
&-\sum_{j=m+2}^N\varepsilon^{2(m+1)(c_{m,j}+1)+j-N-1}S_{c_{m,j}}(\partial u_{m+1},\varepsilon^{m+1})R_{c_{m,j},f_{m,j}}(\omega t,y_{m+1},u_{m+1},\varepsilon^{m+1})\nonumber\\
&+\left(I_n+\varepsilon^{m+1} \partial u_{m+1}\right)^{-1}g_m(\omega t, y_{m+1}+\varepsilon^{m+1} u_{m+1},\varepsilon).
\end{align*}
Similar to the discussion in the first-step transformation, it follows from Cauchy estimate that
$f_{m+1,k}$ ($k=m+2,\dots, N$) and $g_{m+1}$ are bounded on $\mathbb{T}^d_{\tilde{\rho}_{m+1}}\times\mathbf{B}_{\tilde{r}_{m+1}}(U)$.  Based on the analyticity of $Y_j$ ($j=1,\dots, m$), $f_{m,k}$ ($k=m+1,\dots, N$) and $g_m$, as well as the Lipschitz continuity of $g_m$ and $\partial g_m$ with respect to $\varepsilon$, it can be shown that
$g_{m+1}$ and $\partial g_{m+1}$ are also Lipschitz continuous with respect to $\varepsilon$, uniformly on $\mathbb{T}^d_{\tilde{\rho}_{m+1}}\times\mathbf{B}_{\tilde{r}_{m+1}}(U)$.

Based on the above discussion, after $N$ transformations, system \eqref{4} is transformed into system \eqref{4.78}, thus completing the proof of the lemma.
\end{proof}

{\bf Data availability statement.} This work does not have any experimental data.

\section*{Acknowledgment}
We are deeply grateful to the editor and reviewer for their suggestions, which have significantly improved the manuscript.
In particular, based on the reviewer's advice, we developed a higher-order averaging method for response solutions, which has substantially strengthened the paper's results.
The work of the first author was partially supported by the National Natural Science Foundation of China (NSFC)
(No. 12471183; 11901080).  The work of the second author was partially supported by NSFC (No. 12531009; 12471183; 12071175). The work of the third author was  partially supported by NSFC (No. 12225103; 12071065) and Science and Technology Development Plan Project of Jilin Province, China (no. 20260602002RC).

\appendix
\section{Appendix}
The following lemmas from \cite{Jo} form the basis for many estimates in this paper.
\begin{lemma}\label{ALe1}
Let $A_0$ be an $n\times n$ matrix such that the spectrum $\spec(A_0)=\diag\{\lambda_0^{(1)},\dots,\lambda_0^{(n)} \}$ with
$|\lambda_0^{(i)}|>2\mu$,  $|\lambda_0^{(i)}-\lambda_0^{(j)}|>2\mu$ for all $1\leq i, j\leq n$, $i\neq j$, and some $\mu>0$. Let $C_0$ be the matrix such
that $C_0^{-1}A_0C_0=\diag\{\lambda_0^{(1)},\dots,\lambda_0^{(n)} \}$ and $\beta_0=\max\{\|C_0\|, \|C_0^{-1}\| \}$.
Then if $A(\varepsilon)$ satisfies $\|A(\varepsilon)-A_0\|<\frac{2\mu}{(3n-1)\beta_0^2}$ for $|\varepsilon|<\varepsilon_1$, the following hold.
\begin{enumerate}
\item[{\rm(1)}] $\spec(A(\varepsilon))=\diag\{\lambda^{(1)}(\varepsilon),\dots,\lambda^{(n)}(\varepsilon) \}$ with
$|\lambda^{(i)}(\varepsilon)|>\mu$,  $|\lambda^{(i)}(\varepsilon)-\lambda^{(j)}(\varepsilon)|>\mu$ for $1\leq i, j\leq n$, $i\neq j$.
\item[{\rm(2)}] There exists a nonsingular matrix $C(\varepsilon)$ such that  $$C^{-1}(\varepsilon)A(\varepsilon)C(\varepsilon)=\diag\{\lambda^{(1)}(\varepsilon),\dots,\lambda^{(n)}(\varepsilon)\}$$
and which satisfies $\max\{\|C(\varepsilon)\|, \|C^{-1}(\varepsilon)\| \}\leq 2\beta_0$.
\item[{\rm(3)}] If $A(\varepsilon)$ is Lipschitz continuous with constant $\mathcal{L}(A(\varepsilon))$, then
$C(\varepsilon)$, $C^{-1}(\varepsilon)$ and $\lambda^{(j)}(\varepsilon)$ ($1\leq j\leq n$) are all Lipschitz continuous. Moreover, there exist
$\kappa_1=\kappa_1(A_0,\mu,\beta_0)$ and $\kappa_2=\kappa_2(A_0,\mu,\beta_0)$ such that
$$\mathcal{L}(C(\varepsilon))\leq \kappa_1\mathcal{L}(A(\varepsilon)),\ \ \ \mathcal{L}(C^{-1}(\varepsilon))\leq \kappa_1\mathcal{L}(A(\varepsilon)),$$
and $\ \mathcal{L}(\lambda^{(j)}(\varepsilon))\leq \kappa_2\mathcal{L}(A(\varepsilon))$ for $1\leq j\leq n$.
\end{enumerate}
\end{lemma}
\begin{lemma}\label{ALe2}
Let $0<\nu\leq1$, $\tau\geq1$. Then
$$\sum\limits_{k\in\mathbb{Z}^d}e^{-\nu |k|}|k|^\tau\leq \frac{20 d}{3\nu^{d+\tau}}\left(\frac{d+\tau-1}{e}\right)^{d+\tau-1}\sqrt{d+\tau-1}.$$
\end{lemma}
\begin{lemma}\label{ALe3}
Define
$$f(z,\varepsilon)=\sum\limits_{|k|\geq2}a_k(\varepsilon)z^k,\ \ \ \ \ k\in\mathbb{N}^n,$$
and assume that the sum is convergent for all $z\in D=D_1\times\cdots\times D_n\subset \mathbb{C}^n$,
where each $D_j$ is a fixed disk in $\mathbb{C}$. Moreover, suppose that $f$ depends on $\varepsilon$ in a Lipschitz way with
Lipschitz constant $L$. Let $\hat{D}=\hat{D}_1\times\cdots\times \hat{D}_n\subset D$ be such that
${\rm radius}(\hat{D}_j)\leq{\rm radius}({D}_j)\alpha$ for some $0<\alpha<1$. Then if $z\in\hat{D}$, it holds that
\begin{enumerate}
\item[{\rm(1)}] $\|\partial_zf(z,\varepsilon_1)-\partial_zf(z,\varepsilon_2)\|\leq\Delta_1(\alpha)L\alpha|\varepsilon_1-\varepsilon_2|$,
\item[{\rm(2)}] $|f(z,\varepsilon_1)-f(z,\varepsilon_2)|\leq\Delta_2(\alpha)L\alpha^2|\varepsilon_1-\varepsilon_2|$,
\end{enumerate}
where $\Delta_1(\alpha)$ and $\Delta_2(\alpha)$ defined for $\alpha<1$ are continuous and increasing functions.
\end{lemma}
\begin{lemma}\label{ALe4}
Let $p(t,\varepsilon)$ be an analytic quasi-periodic function on a strip of width $\rho$,
$$p(t,\varepsilon)=\sum\limits_{k\in\mathbb{Z}^d}p_k(\varepsilon)e^{\sqrt{-1}\langle k,\omega\rangle t}.$$
Then the following statements hold.
\begin{enumerate}
\item[{\rm(1)}] If $p(t,\varepsilon)$ is Lipschitz continuous in $\varepsilon$ with constant $L$, then each Fourier coefficient $p_k(\varepsilon)$ is Lipschitz continuous in $\varepsilon$ with constant $L_k = L e^{-\rho |k|}$.
\item[{\rm(2)}] Assume each Fourier coefficient $p_k(\varepsilon)$ is Lipschitz continuous in $\varepsilon$ with constant $L_k$, and that $L_k \leq L |k|^{\tau} e^{-\rho |k|}$ for each $k \neq 0$ and some $L > 0$. If $p(t,\varepsilon)$ is restricted to a strip of width $\rho_1 < \rho$, then it is Lipschitz continuous in $\varepsilon$ with constant
$$L' = L_0 + L \varpi(\tau, \rho - \rho_1),$$
where $\varpi$ is defined in \eqref{2.5}.
\end{enumerate}

\end{lemma}

\baselineskip 9pt \renewcommand{\baselinestretch}{1.08}

\end{document}